\documentclass[preprint,3p,fleqn,sort]{elsarticle}
\usepackage[labelfont=bf, singlelinecheck=false,justification=justified]{caption}
\usepackage{xcolor}
\usepackage{amsmath,amssymb,amsthm}
\usepackage{stmaryrd} 
\usepackage{enumitem}
\usepackage{bm,cancel}
\usepackage{accents}
\usepackage{lineno,hyperref}
\usepackage[T1]{fontenc}
\makeatletter
\g@addto@macro\normalsize{%
	\setlength\abovedisplayskip{5pt}
	\setlength\belowdisplayskip{5pt}
	\setlength\abovedisplayshortskip{6pt}
	\setlength\belowdisplayshortskip{6pt}
}
\makeatother

\newcommand{\pderivative}[2]{\frac{\partial #1}{\partial #2}}
\let\rho\varrho
\newcommand{\avg}[1]{\left\{\hspace*{-3pt}\left\{#1\right\}\hspace*{-3pt}\right\}}
\newcommand{\jump}[1]{\ensuremath{\left\llbracket #1 \right\rrbracket}}
\newcommand\iprod[1]{\left\langle #1\right\rangle} 			
\newcommand\iprodN[1]{\left\langle #1\right\rangle_{\!N}}		
\renewcommand\vec[1]{\accentset{\,\rightarrow}{#1}}
\newcommand\spacevec[1]{\accentset{\,\rightarrow}{#1}}		
\newcommand\contravec[1]{\tilde{ #1}}					
\newcommand\contraspacevec[1]{\spacevec{\tilde{#1}}}		
\newcommand\statevec[1]{\mathbf #1}					
\newcommand\statevecGreek[1]{\boldsymbol #1}			
\newcommand\contrastatevec[1]{\tilde{\mathbf #1}} 			
\newcommand\acclrvec[1]{\accentset{\,\leftrightarrow}{#1}}	
\newcommand\bigstatevec[1]{\acclrvec{{\mathbf #1}}}		
\newcommand\biggreekstatevec[1]{\acclrvec{\boldsymbol #1}}	
\newcommand\bigcontravec[1]{\acclrvec{\tilde{\mathbf #1}}} 	
\newcommand\threeMatrix[1]{\underline{ #1}}				
\newcommand\nineMatrix[1]{\mathsf{ #1}}					
\newcommand\twentysevenMatrix[1]{\underline{\mathsf{ #1}}}  
\newcommand\matrixvec[1]{\mathcal #1}					
\newcommand{\JanC}{\statevecGreek{\phi}}				
\newcommand{\JanD}{\statevecGreek{\Phi}}				
\newcommand{\wDotJan}{\theta}						
\newcommand{\dmat}{\matrixvec{D}}						
\newcommand{\qmat}{\matrixvec{Q}}						
\newcommand{\mmat}{\matrixvec{M}}					
\newcommand{\bmat}{\matrixvec{B}}						
\newcommand{\viscosity}{\mu_{\mathrm{NS}}}      	      				
\newcommand{\resistivity}{\mu_{\mathrm{R}}}						
\newcommand{\Bstar}{\left(\JanD\contraspacevec{B}\right){}^{\!\!\Diamond}}		
\newcommand{\testfuncOne}{\statevecGreek{\varphi}}  
\newcommand{\testfuncTwo}{\biggreekstatevec{\vartheta}}
\theoremstyle{plain}
\newtheorem{thm}{Theorem}
\newtheorem{lem}{Lemma}
\newtheorem{cor}{Corollary}
\theoremstyle{remark}
\newtheorem{rem}{Remark}
\newcommand{\dS}{\,\operatorname{dS}}		
\newcommand{\ec}{\mathrm{EC}}			
\newcommand{\es}{\mathrm{ES}}			
\newcommand{\ent}{{\epsilon}} 				
\newcommand{\ma}{-}					
\renewcommand{\sl}{+}					
\newcommand\interiorfaces{{\mathrm{edges}}}
\begin{document}

\begin{frontmatter}

\title{An entropy stable nodal discontinuous Galerkin method for the resistive MHD equations: Continuous analysis and GLM divergence cleaning}
\author[mathematik]{Marvin Bohm\corref{correspondingauthor}}
\cortext[correspondingauthor]{Corresponding author}
\ead{mbohm@math.uni-koeln.de}
\author[mathematik]{Andrew R.~Winters}
\author[physik]{Dominik Derigs}
\author[mathematik]{Gregor J.~Gassner}
\author[physik]{Stefanie Walch}
\author[geophysik]{Joachim Saur}

\address[mathematik]{Mathematisches Institut, Universit\"at zu K\"oln, Weyertal 86-90, 50931 K\"oln}
\address[physik]{I.\,Physikalisches Institut, Universit\"at zu K\"oln, Z\"ulpicher Stra\ss{}e~77, 50937 K\"oln}
\address[geophysik]{Institut f\"ur Geophysik und Meteorologie, Universit\"at zu K\"oln, Pohligstra\ss{}e 3, 50969 K\"oln}

\numberwithin{equation}{section}


\begin{abstract}
This work presents an extension of discretely entropy stable discontinuous Galerkin (DG) methods to the resistive magnetohydrodynamics (MHD) equations. Although similar to the compressible Navier-Stokes equations at first sight, there are some important differences concerning the resistive MHD equations that need special focus. The continuous entropy analysis of the ideal MHD equations, which are the advective parts of the resistive MHD equations, shows that the divergence-free constraint on the magnetic field components must be incorporated as a non-conservative term in a form either proposed by Powell or Janhunen. Consequently, this non-conservative term needs to be discretized, such that the approximation is consistent with the entropy. As an extension of the ideal MHD system, we address in this work the continuous analysis of the resistive MHD equations and show that the entropy inequality holds. Thus, our first contribution is the proof that the resistive terms are symmetric and positive semi-definite when formulated in entropy space as gradients of the entropy variables. Moreover, this enables the construction of an entropy stable DG discretization for the resistive MHD equations. However, the resulting method suffers from large errors in the divergence-free constraint, since no particular treatment of divergence errors is included in the standard resistive MHD model. Hence, our second contribution is the extension of the resistive MHD equations with proper divergence cleaning based on a generalized Lagrange multiplier (GLM) strategy. We construct and analyze a DG method that is entropy stable for the resistive MHD equations and has a built-in GLM divergence cleaning mechanism. The theoretical derivations and proofs are then verified by several numerical examples. Specifically, we demonstrate that the entropy stable method with GLM divergence cleaning is more robust than the standard high-order DG method for the resistive MHD equations.
\end{abstract}

\end{frontmatter}

\section{Introduction}\label{Sec:Intro}

The resistive magnetohydrodynamic (MHD) equations are of great interest in many areas of plasma, space and astrophysics. This stems from a wide range of applications such as electromagnetic turbulence in conducting fluids, magnetically confined fusion for power generation, modeling the action of dynamos and predicting the interaction of the solar wind with planets or moons. The governing equations are able to describe both dense and sparse plasmas that are time-dependent with motions that feature a wide range of temporal and spatial scales, e.g., compressible MHD turbulence. In addition, the resistive MHD equations exhibit a mixed hyperbolic-parabolic character depending on the strength of the viscous or resistive effects. Another important property, in a closed physical system is the second law of thermodynamics, i.e., the evolution of the entropy. In the absence of resistivity and viscosity, that is for the ideal MHD model, and if the resulting solutions are smooth, the entropy of the system is an additional conserved quantity, although not explicitly built into the mathematical model. Further, in the presence of shocks, the second law of thermodynamics becomes the entropy inequality, e.g. \cite{harten1983}, which guarantees that entropy is always dissipated with the correct sign. It is assumed that the additional resistive terms have a pure entropy dissipative effect as well. But, to the best of our knowledge, no continuous entropy analysis of the resistive MHD equations has been presented in literature and it is unclear if the entropy inequality holds for the full resistive MHD equations. 

The discrete analog to the entropy inequality of the continuous partial differential equations (PDE) is \textit{entropy stability}. Constructing such entropy stable methods for the advective parts of the MHD model has been studied by many authors, e.g., \cite{Barth1999,Chandrashekar2015,Liu2017,Rossmanith2013,Winters2016}. A complication when discussing entropy stability for MHD models is the necessity of the involution condition, that is, the divergence-free constraint of the magnetic field \cite{Barth1999,Godunov1972}
\begin{equation}\label{eq:divBIntro}
\spacevec{\nabla}\cdot\spacevec{B} = 0.
\end{equation}
The condition \eqref{eq:divBIntro}, like entropy stability, is an additional PDE not explicitly built into the resistive MHD equations. To cancel extraneous terms in the entropy analysis requires the addition of a non-conservative term that is proportional to the divergence of the magnetic field, as first introduced by Godunov \cite{Godunov1972}. In the context of numerical schemes, several such non-conservative terms have been proposed over the years by Brackbill and Barnes \cite{Brackbill1980}, Powell \cite{Powell1999} and Janhunen \cite{Janhunen2000}. On the continuous level, adding a non-conservative term scaled by \eqref{eq:divBIntro} is a clever way of adding zero to the model. However, for numerical approximations, there are known stability and accuracy issues that differ between the three types of non-conservative terms \cite{Sjogreen2017}.

Such differences in stability and accuracy are largely due to the discrete satisfaction of the divergence-free condition. It is well known in the MHD numerics community that even if the initial conditions of a problem satisfy \eqref{eq:divBIntro}, it is not guaranteed that the discrete evolution of a plasma will remain divergence-free in the magnetic field. Therefore, many numerical techniques have been devised to control errors introduced into the divergence-free constraint by a numerical approximation. These include adding the aforementioned non-conservative term, the projection approach of Brackbill and Barnes \cite{Brackbill1980}, the method of constrained transport introduced by Evans and Hawley \cite{Evans1988} and the generalized Lagrange multiplier (GLM) hyperbolic divergence cleaning technique originally proposed for the ideal MHD equations by Dedner et al. \cite{Dedner2002}. A thorough review of the behavior and efficacy of  these techniques, except hyperbolic divergence cleaning, is provided by T\'{o}th \cite{Toth2000}. Due to its relative ease of implementation and computational efficiency we are most interested in the method of hyperbolic divergence cleaning. However, the current work is also concerned with constructing entropy stable numerical approximations. Recent work by Derigs et al. \cite{Derigs2017} modified the additional GLM divergence cleaning system in such a way that the resulting ideal GLM-MHD system is consistent with the continuous entropy analysis and provides inbuilt divergence cleaning capabilities. 

The goal of this work is to construct an entropy stable high-order discontinuous Galerkin (DG) method for the modified GLM-MHD system augmented with resistive terms to obtain a \textit{robust}, high-order accurate numerical discretization with inbuilt divergence cleaning. Recently, the construction and analysis of discretely entropy stable DG methods for non-linear conservation laws have made major progress \cite{carpenter_esdg,Gassner:2016ye,Liu2017}. The key to \textit{discrete} entropy stability is to mimic the integration-by-parts property in the DG operators with the so-called summation-by-parts (SBP) property. This enables the construction of DG methods that are entropy stable without the assumption of exact evaluation of the variational forms. The reason we are particularly interested in entropy stability for our DG method is the expected increase in robustness.  We know that for the compressible Navier-Stokes equations an entropy stable DG method has drastically increased robustness in comparison to the standard DG and even in comparison to DG methods with de-aliasing, e.g. \cite{carpenter_esdg,Gassner:2016ye}. Solutions of the resistive MHD equations can be even more complicated than solutions to the compressible Navier-Stokes equations. As such, we aim for the construction of a high-order method that is able to capture the wide range of scales in a robust and stable way.  

However, the construction of entropy stable high-order DG methods for the resistive MHD equations involves several key differences in comparison to standard conservation laws. One major difference is the addition of a non-conservative PDE term proportional to the discrete divergence of the magnetic field. Typically, entropy stable discontinuous Galerkin spectral element methods (DGSEM) are built for conservation laws. Thus, precise and careful numerical treatment of the non-conservative term is necessary. We note that there is recent work on entropy stable DG methods applied to the ideal MHD equations by Rossmanith \cite{Rossmanith2013}, Liu et al. \cite{Liu2017} and Gallego-Valencia \cite{Valencia2017}. In particular, we extend the ideas presented in Liu et al. \cite{Liu2017} and provide clarification and motivation to the structure of the presented discretization of the non-conservative term and add detailed proofs of the 2D extension. As already mentioned, another important difference is the numerical control of the divergence-free constraint. To do so, we incorporate the modified GLM divergence cleaning approach in our full high-order DG entropy analysis. Furthermore, we build on the recent work of Gassner et al. \cite{Gassner2017} who showed that using a Bassi-Rebay (BR1) type scheme \cite{Bassi&Rebay:1997:B&F97} is entropy stable for the viscous terms of the compressible Navier-Stokes equations, if the scheme is formulated with the gradients of the entropy variables instead of the gradients of the conservative or primitive variables. To use this proof, we reformulate the resistive terms by the gradients of the entropy variables and show that the resulting resistive coefficient matrices are symmetric and positive semi-definite.  

The remainder of this paper is organized as follows: Our first main result, in Sec. \ref{Sec:Systems}, is the continuous entropy analysis of the three-dimensional resistive GLM-MHD equations, which demonstrates that the model indeed satisfies the entropy inequality and that the resistive terms can be recast into a symmetric and positive semi-definite form. The discontinuous Galerkin spectral element method (DGSEM) is outlined in Sec. \ref{Sec:DG}. The second main result of this work, in Sec. \ref{Sec:ESDGres}, is to demonstrate the entropy stability of the numerical approximation by discretely mimicking the continuous entropy analysis with special attention given to the GLM divergence cleaning and resistive parts. We present numerical results in Sec. \ref{Sec:Num} to verify the high-order nature and theoretical entropy conservation (without viscous terms) and entropy stability (including viscous terms) of the derived scheme. We furthermore demonstrate the increased robustness of the novel DG method. Conclusion remarks are given in the final section.

\section{Continuous entropy analysis}\label{Sec:Systems}

As a starting point for the discussion and numerical approximations in this work we outline the resistive GLM-MHD equations. In general, we consider systems of conservation laws in a domain $\Omega \subset \mathbb{R}^3$ defined as
\begin{equation}
 \statevec{u}_t + \vec{\nabla} \cdot \bigstatevec{f} = \statevec{0},
\label{consLaw}
\end{equation}  
where $\statevec{u}$ denotes the vector of conserved variables and $\bigstatevec{f}$ the multidimensional flux vector. These definitions allow for a compact notation that will simplify the analysis, i.e., we define block vectors with the double arrow as
\begin{equation}\label{eq:firstBlockVector}
\bigstatevec{f} =
 \left[ {\begin{array}{*{20}{c}}
  {{\statevec f_1}} \\ 
  {{\statevec f_2}} \\ 
  {{\statevec f_3}} 
\end{array}} 
\right],
\end{equation}
and the spatial gradient of a state as
\begin{equation}
\spacevec{\nabla} \statevec u = \left[ {\begin{array}{*{20}{c}}
\statevec{u}_x\\
\statevec{u}_y\\
\statevec{u}_z
\end{array}}
\right].
\end{equation}
The gradient of a spatial vector is a second order tensor, written in matrix form as
\begin{equation}
\spacevec{\nabla}\spacevec{v} =\left[ {\begin{array}{*{20}{ccc}}
\pderivative{v_1}{x} & \pderivative{v_1}{y} & \pderivative{v_1}{z} \\[0.2cm]
\pderivative{v_2}{x} & \pderivative{v_2}{y} & \pderivative{v_2}{z} \\[0.2cm]
\pderivative{v_3}{x} & \pderivative{v_2}{y} & \pderivative{v_3}{z} \\[0.2cm]
\end{array}}
\right]
\end{equation}
and the divergence of a flux written as a block vector is defined as
\begin{equation}
\spacevec\nabla  \cdot \bigstatevec f = \left(\statevec{f}_1\right)_{\!x} + \left(\statevec{f}_2\right)_{\!y} + \left(\statevec{f}_3\right)_{\!z}.
\end{equation}

\subsection{Resistive GLM-MHD equations}\label{Sec:ResistiveMHD}

The equations that govern the evolution of resistive, conducting fluids depend on the solution as well as its gradient \cite{dumbser2009,fambri2017,Warburton1999}. In the advective components we also include the recent modifications proposed by Derigs et al. \cite{Derigs2017} to incorporate a GLM divergence cleaning framework into the model. Thus, the complete mathematical model considered in this work can be defined in the domain $\Omega \subset \mathbb{R}^3$ by
\begin{equation}
 \statevec{u}_t + \vec{\nabla} \cdot \bigstatevec{f}^a(\statevec{u}) - \vec{\nabla} \cdot \bigstatevec{f}^v(\statevec{u},\vec{\nabla} \statevec{u})  + \JanC(\vec{\nabla}\cdot \spacevec{B}) = \statevec{r}
\label{resGLMMHD}
\end{equation} 
with $\statevec{u} = (\rho, \rho \spacevec{v}, E, \spacevec{B}, \psi )^T$ as well as the advective and viscous fluxes:
\begin{equation}
\resizebox{0.925\hsize}{!}{$
\bigstatevec{f}^a(\statevec{u}) = 
\begin{pmatrix} 
\rho \spacevec{v} \\[0.15cm]
\rho (\spacevec{v}\, \spacevec{v}^{\,T}) + \left(p+\frac{1}{2} \|\spacevec{B}\|^2\right) \threeMatrix{I} - \spacevec{B} \spacevec{B}^T \\[0.15cm]
\spacevec{v}\left(\frac{1}{2}\rho \left\|\spacevec{v}\right\|^2 + \frac{\gamma p}{\gamma -1} + \|\spacevec{B}\|^2\right) - \spacevec{B}\left(\spacevec{v}\cdot\spacevec{B}\right) + c_h \psi \spacevec{B} \\[0.15cm]
\spacevec{v}\,\spacevec{B}^T - \spacevec{B}\,\spacevec{v}^{\,T} + c_h \psi \threeMatrix{I} \\ c_h \spacevec{B}\\[0.15cm]
\end{pmatrix} 
~~~~~ 
\bigstatevec{f}^v(\statevec{u},\vec{\nabla} \statevec{u})  = 
\begin{pmatrix} 
0 \\[0.15cm] 
\threeMatrix{\tau} \\[0.15cm]
\threeMatrix{\tau}\spacevec{v} -\vec{\nabla} T - \resistivity \left(\left(\vec{\nabla} \times \spacevec{B}\right) \times \spacevec{B}\right) \\[0.15cm]
\resistivity \left(\left(\vec{\nabla} \spacevec{B}\right)^T - \vec{\nabla} \spacevec{B}\right)\\[0.15cm]
0 \\[0.15cm]
\end{pmatrix}$}
\label{fluxes}
\end{equation}
Here, $\rho,\,\spacevec{v}=(v_1,v_2,v_3)^T,\,p,\,E$ are the mass density, fluid velocities, pressure and total energy, respectively, $\spacevec{B} = (B_1,B_2,B_3)^T$ denotes the magnetic field components and $\threeMatrix{I}$ the $3\times 3$ identity matrix. Furthermore, the viscous stress tensor reads \cite{kundu2008}
\begin{equation}
\threeMatrix{\tau} = \viscosity ((\vec{\nabla} \spacevec{v}\,)^T + \vec{\nabla} \spacevec{v}\,) - \frac{2}{3} \viscosity(\vec{\nabla} \cdot \spacevec{v}\,) \threeMatrix{I}
\label{stress}
\end{equation}
and the heat flux is defined as
\begin{equation}
\vec{\nabla} T = -\kappa \vec{\nabla} \left(\frac{p}{R \rho}\right).
\label{heat}
\end{equation}
The introduced constants $\viscosity,\resistivity,\kappa,R > 0$ describe the viscosity from the Navier-Stokes equations, resistivity of the plasma, thermal conductivity and the universal gas constant, respectively. In particular, the constants $\viscosity$ and $\resistivity$ are first-order transport coefficients that describe the kinematic viscosity and the diffusivity of the magnetic field \cite{woods1993}. We close the system with the ideal gas assumption, which relates the total energy and pressure 
\begin{equation}
p = (\gamma-1)\left(E - \frac{1}{2}\rho\left\|\spacevec{v}\right\|^2 - \frac{1}{2}\|\spacevec{B}\|^2 - \frac{1}{2}\psi^2\right),
\label{eqofstate}
\end{equation}
where $\gamma$ denotes the adiabatic coefficient. 

The system \eqref{resGLMMHD} contains the GLM extension indicated by the additional field variable $\psi$, which controls the divergence error and propagates it out of the physical domain by the constant wave speed $c_h$. We also introduce two non-conservative terms: The first one is $\JanC$ scaled by $\vec{\nabla} \cdot \spacevec{B}$ and, thus, zero in the continuous case. The particular choice of $\JanC$ is related to the thermodynamic properties of \eqref{resGLMMHD} as it ensures the entropy conservation for the advective part of the system outlined in the next section. In the original approach to derive thermodynamically consistent methods for the ideal MHD equations $\JanC$ is chosen to be the Powell source term \cite{Barth1999,Powell1999}, which also recovers the symmetrization of the system \cite{Godunov1972}. In this work we restrict to the alternative Janhunen source term \cite{Janhunen2000}, for two reasons:
\begin{enumerate}
\item It is sufficient to restore the satisfaction of the second law of thermodynamics even when $\spacevec{\nabla}\cdot\spacevec{B}\ne 0$, e.g. \cite{Winters2016}, which is important in the construction of entropy stable numerical methods.
\item The equations remain conservative in all the hydrodynamic variables so the satisfaction of the first law of thermodynamics, i.e. conservation of the total energy, is unchanged.
\end{enumerate}
In particular, the components of the Janhunen source term are
\begin{equation}\label{Janhunen}
\JanC = \left(0,0,0,0,0,v_1,v_2,v_3,0\right)^T.
\end{equation}
We note that because on the continuous level $\vec{\nabla}\cdot\spacevec{B}=0$ the modification of $\JanC$ to the resistive GLM-MHD system is consistent.

The second source term $\statevec{r}$ on the right hand side of the system solely provides additional damping of the divergence error \cite{Dedner2002,Derigs2017}, if desired, and is
\begin{equation}\label{damping}
	\statevec{r} = \left(0,0,0,0,0,0,0,0,-\alpha\psi\right)^T
\end{equation}
with $\alpha\geq 0$. 

\subsection{Thermodynamic properties of the system}\label{Sec:ContEntr}

To discuss the thermodynamic properties of the resistive GLM-MHD equations \eqref{resGLMMHD} we translate the concepts of the first and second law of thermodynamics into a mathematical context. To do so, we first exclusively examine the advective and non-conservative term proportional to \eqref{eq:divBIntro}. The ideal GLM-MHD equations satisfy the first law of thermodynamics, because the evolution of the total fluid energy is one of the conserved quantities. This is true for any choice of the vector $\JanC$ because \eqref{eq:divBIntro} is assumed to hold in the continuous analysis. But, on the discrete level, this is not the case as noted by many authors \cite{Derigs2017,Powell1999,Janhunen2000,Toth2000}. However, the mathematical description of the second law of thermodynamics is more subtle, because the entropy is not explicitly built into the system. Thus, we require more formalism and utilize the well-developed entropy analysis tools for hyperbolic systems, e.g. \cite{harten1983,tadmor1984,mock1980}. As such, we define a strongly convex entropy function that is then used to define an injective mapping between state space and entropy space. Note that we adopt the mathematical notation that the entropy function is negative as is often done, e.g. \cite{harten1983,Tadmor2003}. Once the advective terms are accounted for in entropy space, we demonstrate the first main result of this work: the resistive terms do not violate the second law of thermodynamics. 

For the ideal and the resistive GLM-MHD equations a suitable entropy function is the entropy density divided by the constant $(\gamma-1)$ for convenience as
\begin{equation}
S(\statevec{u}) = - \frac{\rho s}{\gamma-1} ~~ \text{ with } ~~ s = \ln\left(p \rho^{-\gamma}\right),
\label{entropy}
\end{equation}
where $s$ is the physical entropy \cite{Landau1959}.  From the entropy function we define the entropy variables to be
\begin{equation}
\statevec{w} = \frac{\partial S}{\partial \statevec{u}} = \left(\frac{\gamma-s}{\gamma-1} - \beta\left\|\spacevec{v}\right\|^2,~2\beta v_1,~2\beta v_2,~2\beta v_3,~-2\beta,~2\beta B_1,~2\beta B_2,~2\beta B_3,~2\beta \psi\right)^T
\label{entrvars}
\end{equation}
and we introduce $\beta = \frac{\rho}{2p}$, which is proportional to the inverse temperature.

It is known, that for smooth solutions when we contract the ideal GLM-MHD equations without viscous fluxes nor any source term $\statevec{r}$ on the right hand side by the entropy variables \eqref{entrvars} we obtain the entropy conservation law \cite{Derigs2017}
\begin{equation}\label{eq:PDEcontracted}
 \frac{\partial S}{\partial t} + \vec{\nabla}\cdot\spacevec{f}^\ent = 0,
\end{equation}
where we define the entropy fluxes to be
\begin{equation}
\spacevec{f}^\ent = \spacevec{v}S.
\end{equation}

Additionally, as it will be necessary in later derivations and the proof of discrete entropy stability, we compute the entropy flux potential to be
\begin{equation}\label{eq:entPotential}
\spacevec{\Psi} = \statevec{w}^T\bigstatevec{f}^a - \spacevec{f}^\ent+\wDotJan \spacevec{B},
\end{equation}
where we introduce notation for the non-conservative state vector $\JanC$ \eqref{Janhunen} contracted into entropy space
\begin{equation}\label{eq:contractSource}
\wDotJan = \statevec{w}^T\JanC = 2\beta(\spacevec{v}\cdot\spacevec{B}). 
\end{equation}
We note that, if the Powell non-conservative state vector \cite{Powell1999} was used, the result \eqref{eq:contractSource} still holds \cite{Derigs2017}.

Furthermore, e.g. in case of shock discontinuities, the solution satisfies the following entropy inequality
\begin{equation} 
  S_t + \vec{\nabla}\cdot(\spacevec{v}S) \leq 0,
\label{entrineq}
\end{equation} 
which is the mathematical description of the second law of thermodynamics for the ideal MHD equations. Next, we account for the resistive terms to demonstrate the entropy inequality for the resistive GLM-MHD equations. To do so, we require a suitable representation to discuss how the resistive terms affect \eqref{entrineq}.  

\begin{lem}[Entropy representation of viscous fluxes]\label{lemma1}~\newline
The viscous fluxes of the resistive GLM-MHD equations in \eqref{fluxes} can be expressed by gradients of the entropy variables as
\begin{equation}
 \bigstatevec{f}^v(\statevec{u},\vec{\nabla} \statevec{u}) = \twentysevenMatrix{K}  \vec{\nabla} \statevec{w}
\label{Kvisc}
\end{equation}
with a block matrix $\twentysevenMatrix{K}\in\mathbb{R}^{27\times 27}$ that is symmetric and positive semi-definite matrix, i.e,
\begin{equation}
\mathsf{q}^T\twentysevenMatrix{K}\mathsf{q}\geq 0,\quad\forall\mathsf{q}\in\mathbb{R}^{27}.
\end{equation}
\end{lem}

\begin{proof}
First, we consider the viscous fluxes of the resistive GLM-MHD system in \eqref{fluxes} separately
\begin{equation}
\bigstatevec{f}^v(\statevec{u},\vec{\nabla} \statevec{u}) = \left[\statevec{f}_1^v, \statevec{f}_2^v, \statevec{f}_3^v \right]^T
\end{equation}
with
\begin{equation*}
\begin{aligned}
\statevec{f}_1^v & = 
\left(
\begin{array}{c} 0 \\[0.1cm] 
\viscosity \left(2(v_{1})_x-\frac{2}{3} \left( (v_{1})_x+(v_{2})_y+(v_{3})_z \right)\right) \\[0.1cm]
\viscosity \left( (v_{1})_y + (v_{2})_x\right) \\ \viscosity \left( (v_{1})_z + (v_{3})_x\right) \\[0.1cm]
 C_1 \\[0.1cm]
  0 \\[0.1cm]
\resistivity \left(  (B_{2})_x - (B_{1})_y\right) \\[0.1cm]
\resistivity \left(  (B_{3})_x - (B_{1})_z\right) \\[0.1cm]
0\\[0.1cm]
\end{array}
\right)
 ~~~~~ 
 \statevec{f}_2^v = 
 \left(
\begin{array}{c} 0 \\[0.1cm]
\viscosity \left( (v_{1})_y + (v_{2})_x\right) \\[0.1cm]
\viscosity\left (2 (v_{2})_y - \frac{2}{3} \left( (v_{1})_x + (v_{2})_y + (v_{3})_z \right)\right) \\[0.1cm]
\viscosity \left( (v_{2})_z + (v_{3})_y \right) \\[0.1cm]
C_2  \\[0.1cm]
\resistivity  \left( (B_{1})_y - (B_{2})_x \right) \\[0.1cm] 
0 \\[0.1cm]
\resistivity \left(  (B_{3})_y-  (B_{2})_z\right)  \\[0.1cm]
0 \\[0.1cm]
\end{array} 
\right)
\\[0.125cm]
&\qquad\qquad\qquad\qquad\qquad\quad
\statevec{f}_3^v = 
\left(
\begin{array}{c} 0 \\[0.1cm] 
\viscosity \left( (v_{1})_z + (v_{3})_x \right) \\[0.1cm] 
\viscosity \left( (v_{2})_z + (v_{3})_y \right) \\[0.1cm]
\viscosity\left (2 (v_{3})_z - \frac{2}{3} \left( (v_{1})_x + (v_{2})_y + (v_{3})_z \right)\right) \\[0.1cm]
C_3 \\[0.1cm]
\resistivity\left( (B_{1})_z - (B_{3})_x \right) \\[0.1cm]
\resistivity\left( (B_{2})_z - (B_{3})_y \right) \\[0.1cm]
0 \\[0.1cm]
0 \\[0.1cm]
\end{array} 
\right)
\end{aligned}
\end{equation*}
where
\begin{align*}
C_1  = &\viscosity\left (v_1\left(2(v_{1})_x - \frac{2}{3} \left((v_{1})_x + (v_{2})_y + (v_{3})_z \right)\right)+ v_2\left( (v_{2})_x + (v_{1})_y \right)+ v_3\left((v_{3})_x + (v_{1})_z \right)\right)+\frac{\kappa}{R}\left(\frac{p}{\rho}\right)_{\!\!x} \\
&+\resistivity \left ( B_2\left( (B_{2})_x - (B_{1})_y \right) +  B_3\left( (B_{3})_x - (B_{1})_z \right) \right), \\[0.1cm]
C_2 = &\viscosity\left (v_2\left(2(v_{2})_y - \frac{2}{3} \left( (v_{1})_x + (v_{2})_y + (v_{3})_z \right)\right)+  v_1\left( (v_{2})_x + (v_{1})_y \right)+  v_3\left( (v_{3})_y + (v_{2})_z \right)\right)+\frac{\kappa}{R}\left(\frac{p}{\rho}\right)_{\!\!y}\\
&+\resistivity \left ( B_1\left( (B_{1})_y - (B_{2})_x \right) +  B_3\left( (B_{3})_y - (B_{2})_z \right)\right), \\[0.1cm]
C_3 = &\viscosity\left(v_3\left(2(v_{3})_z - \frac{2}{3} \left( (v_{1})_x + (v_{2})_y + (v_{3})_z \right)\right)+ v_1\left( (v_{1})_z + (v_{3})_x \right)+ v_2\left( (v_{2})_z + (v_{3})_y \right)\right)+\frac{\kappa}{R}\left(\frac{p}{\rho}\right)_{\!\!z} \\[0.1cm]
&+\resistivity \left ( B_1\left( (B_{1})_z - (B_{3})_x \right) +B_2\left( (B_{2})_z - (B_{3})_y \right)\right).
\end{align*}
Using the vector of entropy variables from \eqref{entrvars}
\begin{equation*}
\statevec{w} = \left(w_1, \ldots , w_9 \right)^T = \left(\frac{\gamma-s}{\gamma-1} - \beta\|\spacevec{v}\|^2,~2\beta v_1,~2\beta v_2,~2\beta v_3,~-2\beta,~2\beta B_1,~2\beta B_2,~2\beta B_3,~2\beta \psi\right)^T
\end{equation*}
we find the following relations:
\begin{align*}
\vec{\nabla} v_1 & = - \frac{1}{w_5} \vec{\nabla} w_2 +  \frac{w_2}{w_5^2} \vec{\nabla} w_5 ~~~~~~ \vec{\nabla} v_2 = - \frac{1}{w_5} \vec{\nabla} w_3 +  \frac{w_3}{w_5^2} \vec{\nabla} w_5 ~~~~~~ \vec{\nabla} v_3 = - \frac{1}{w_5} \vec{\nabla} w_4 +  \frac{w_4}{w_5^2} \vec{\nabla} w_5 \\
\vec{\nabla} B_1 & = - \frac{1}{w_5} \vec{\nabla} w_6 +  \frac{w_6}{w_5^2} \vec{\nabla} w_5 ~~~~ \vec{\nabla} B_2 = - \frac{1}{w_5} \vec{\nabla} w_7 +  \frac{w_7}{w_5^2} \vec{\nabla} w_5 ~~~~~ \vec{\nabla} B_3 = - \frac{1}{w_5} \vec{\nabla} w_8 +  \frac{w_8}{w_5^2} \vec{\nabla} w_5 \\
\vec{\nabla} \left(\frac{p}{\rho}\right) & = \frac{1}{w_5^2} \vec{\nabla} w_5
\end{align*}
With some algebraic effort we can determine the matrices $\nineMatrix{K}_{ij} \in \mathbb{R}^{9\times 9}, (i,j=1,2,3)$ to express the viscous fluxes in terms of matrices times the gradients of entropy variables:
\begin{align}
\statevec{f}_1^v & = \nineMatrix{K}_{11} \frac{\partial \statevec{w}}{\partial x} + \nineMatrix{K}_{12} \frac{\partial \statevec{w}}{\partial y} + \nineMatrix{K}_{13} \frac{\partial \statevec{w}}{\partial z} \label{fviscK1}\\
\statevec{f}_2^v & = \nineMatrix{K}_{21} \frac{\partial \statevec{w}}{\partial x} + \nineMatrix{K}_{22} \frac{\partial \statevec{w}}{\partial y} + \nineMatrix{K}_{23} \frac{\partial \statevec{w}}{\partial z} \\
\statevec{f}_3^v & = \nineMatrix{K}_{31} \frac{\partial \statevec{w}}{\partial x} + \nineMatrix{K}_{32} \frac{\partial \statevec{w}}{\partial y} + \nineMatrix{K}_{33} \frac{\partial \statevec{w}}{\partial z} \label{fviscK3}
\end{align}
We collect all these $9\times 9$ block matrices into the matrix $\twentysevenMatrix{K}\in\mathbb{R}^{27\times 27}$
\begin{equation}\label{eq:matrixK}
\twentysevenMatrix{K} = \begin{pmatrix}
\nineMatrix{K}_{11} & \nineMatrix{K}_{12} & \nineMatrix{K}_{13} \\[0.1cm]
\nineMatrix{K}_{21} & \nineMatrix{K}_{22} & \nineMatrix{K}_{23} \\[0.1cm]
\nineMatrix{K}_{31} & \nineMatrix{K}_{32} & \nineMatrix{K}_{33} \\[0.1cm]
\end{pmatrix},
\end{equation}
which clearly yields
\begin{equation}
 \bigstatevec{f}^v = \bigstatevec{f}^v(\statevec{u},\vec{\nabla} \statevec{u}) = \twentysevenMatrix{K}  \vec{\nabla} \statevec{w}.
\label{Kgradient}
\end{equation}
For clarification, we present the first matrix
\begin{equation}\label{eq:matK11}
\resizebox{0.925\hsize}{!}{$
\nineMatrix{K}_{11} = \frac{1}{w_5}\begin{pmatrix}
0 & 0 & 0 & 0 & 0 & 0 & 0 & 0 & 0 \\[0.15cm]
0 & -\frac{4\resistivity}{3} & 0 & 0 & \frac{4\resistivity w_2}{3w_5} & 0 & 0 & 0 & 0 \\[0.15cm]
0 & 0 & -\viscosity & 0 & \frac{\viscosity w_3}{w_5} & 0 & 0 & 0 & 0 \\[0.15cm]
0 & 0 & 0 & -\viscosity & \frac{\viscosity w_4}{w_5} & 0 & 0 & 0 & 0 \\[0.15cm]
0 & \frac{4\viscosity w_2}{3w_5} & \frac{\viscosity w_3}{w_5} & \frac{\viscosity w_4}{w_5} & -\frac{4\viscosity w_2^2}{3w_5^2}-\frac{\viscosity w_3^2}{w_5^2}-\frac{\viscosity w_4^2}{w_5^2}+\frac{\kappa}{Rw_5}-\frac{\resistivity w_7^2}{w_5^2}-\frac{\resistivity w_8^2}{w_5^2} & 0 & \frac{\resistivity w_7}{w_5} & \frac{\resistivity w_8}{w_5} & 0 \\[0.15cm]
0 & 0 & 0 & 0 & 0 & 0 & 0 & 0 & 0 \\[0.15cm]
0 & 0 & 0 & 0 & \frac{\resistivity w_7}{w_5} & 0 & -\resistivity & 0 & 0 \\[0.15cm]
0 & 0 & 0 & 0 & \frac{\resistivity w_8}{w_5} & 0 & 0 & -\resistivity & 0 \\[0.15cm]
0 & 0 & 0 & 0 & 0 & 0 & 0 & 0 & 0 \\[0.15cm]
\end{pmatrix}.$}
\end{equation}
The other matrices $\nineMatrix{K}_{12},\ldots,\nineMatrix{K}_{33}$ are explicitly stated in \ref{Sec:DisMatrix}. It is straightforward to verify that the matrix $\twentysevenMatrix{K}$ is symmetric by inspecting the block matrices listed in \eqref{eq:matK11} and \eqref{eq:matK12} - \eqref{eq:matK33} where the following relationships hold
\begin{equation}\label{eq:symmMatrix}
\nineMatrix{K}_{11} = \nineMatrix{K}_{11}^T,\;\nineMatrix{K}_{22} = \nineMatrix{K}_{22}^T,\;\nineMatrix{K}_{33} = \nineMatrix{K}_{33}^T,\;\nineMatrix{K}_{12} = \nineMatrix{K}_{21}^T,\;\nineMatrix{K}_{13} = \nineMatrix{K}_{31}^T,\;\nineMatrix{K}_{23} = \nineMatrix{K}_{32}^T.
\end{equation}

To show that the matrix $\twentysevenMatrix{K}$ is positive semi-definite is more involved. We first note that it is possible to split the matrix \eqref{eq:matrixK} into the viscous terms associated with the Navier-Stokes equations and the resistive terms of the magnetic fields that arise in the resistive GLM-MHD equations. We exploit this fact and rewrite the total diffusion matrix into two pieces
\begin{equation}
\twentysevenMatrix{K} = \twentysevenMatrix{K}^{\text{NS}} + \twentysevenMatrix{K}^{\text{RMHD}},
\end{equation}
where all terms with $\viscosity$ are put in $\twentysevenMatrix{K}^{\text{NS}}$ and all terms with $\resistivity$ are in $\twentysevenMatrix{K}^{\text{RMHD}}$. 
It is easy to verify that the NS and RMHD block matrices are symmetric, as both satisfy \eqref{eq:symmMatrix}. A further convenience is that the Navier-Stokes part, $\twentysevenMatrix{K}^{\text{NS}}$, is known to be positive semi-definite \cite{Dutt1988}
\begin{equation}
\mathsf{q}^T\twentysevenMatrix{K}^{\text{NS}}\mathsf{q}\geq 0,\quad \forall\mathsf{q}\in\mathbb{R}^{27}.
\end{equation}
Thus, all that remains is to demonstrate that the additional resistive dissipation matrix, $\twentysevenMatrix{K}^{\text{RMHD}}$, is positive semi-definite. To do so we examine the eigenvalues of the system. We use the computer algebra system Maxima \cite{maxima} to find an explicit expression of the eigenvalues to be
\begin{equation}\label{eq:newEVs}
\lambda^{\text{RMHD}}_0 = 0,\;\lambda^{\text{RMHD}}_1 = \frac{2\resistivity p}{\rho},\;\lambda^{\text{RMHD}}_2 = \frac{\resistivity p\left(\|\spacevec{B}\|^2+2\right)}{\rho},\qquad\textrm{multiplicity:}\;\{24,1,2\}.
\end{equation}
Under the physical assumptions that $p,\rho>0$ and $\resistivity\geq 0$ we see that the eigenvalues \eqref{eq:newEVs} of the matrix $\twentysevenMatrix{K}^{\text{RMHD}}$ are all non-negative. Hence, the block matrix $\twentysevenMatrix{K}$ is symmetric and positive semi-definite. 
\end{proof}

With the ability to rewrite the viscous fluxes as a linear combination of the entropy variable gradients, we present our first main result:

\begin{thm}[Entropy inequality for the resistive GLM-MHD equations]\label{theorem:1}~\newline
The resistive GLM-MHD equations \eqref{resGLMMHD}, when contracted by the entropy variables, satisfy the entropy inequality \eqref{entrineq}, provided we choose $\JanC$ to be the Powell \cite{Powell1999} or Janhunen non-conservative term \eqref{Janhunen} and $\alpha\geq 0$ in \eqref{damping}.
\end{thm}
\begin{proof}
We start by contracting the resistive GLM-MHD system \eqref{resGLMMHD} with the entropy variables and integrate over the domain: 
\begin{equation}
 \iprod{\statevec{u}_t, \statevec{w}} + \iprod{\vec{\nabla} \cdot \bigstatevec{f}^a(\statevec{u}) + \JanC(\vec{\nabla}\cdot \spacevec{B}),\statevec{w}} = \iprod{\vec{\nabla} \cdot \bigstatevec{f}^v(\statevec{u},\vec{\nabla} \statevec{u}),\statevec{w}} + \iprod{\statevec{r},\statevec{w}}.
\label{weakformES}
\end{equation} 
Here, $\left\langle \cdot, \cdot \right\rangle$ denotes the $L^2(\Omega)$ inner product, e.g.,
\begin{equation}\label{eq:L2IP}
\iprod{\statevec{u}_t,\statevec{w}} = \int\limits_{\Omega}\statevec{w}^T\statevec{u}_t\,\textrm{dV}.
\end{equation}
From the definition of the entropy variables we have
\begin{equation}
\statevec{w}^T\!\statevec{u}_t = \left(\frac{\partial S}{\partial\statevec{u}}\right)^T\!\!\statevec{u}_t  = S_t.
\label{eq:wsContraction}
\end{equation}
Next, for clarity, we separate the advective flux into Euler, ideal MHD and GLM parts
\begin{equation}
\bigstatevec{f}^a(\statevec{u}) = \bigstatevec{f}^{a,\text{Euler}}+\bigstatevec{f}^{a,\text{MHD}}+\bigstatevec{f}^{a,\text{GLM}}.
\end{equation}
The Euler terms generate the divergence of the entropy flux, e.g., \cite{harten1983}
\begin{equation}\label{eq:fluxContractionEuler}
\statevec{w}^T\!\left( \spacevec{\nabla}\cdot\bigstatevec{f}^{a,\text{Euler}}\right)= \spacevec\nabla\cdot\spacevec{f}^\ent,
\end{equation}
the ideal MHD and non-conservative term cancel, e.g., \cite{Barth1999,Liu2017}
\begin{equation}\label{eq:fluxContractionMHD}
\statevec{w}^T\!\left( \spacevec{\nabla}\cdot\bigstatevec{f}^{a,\text{MHD}} + \JanC(\spacevec{\nabla}\cdot\spacevec{B})\right)= 0
\end{equation}
and the GLM terms vanish directly with the modifications introduced in \cite{Derigs2017}
\begin{equation}\label{eq:fluxContractionGLM}
\statevec{w}^T\!\left( \spacevec{\nabla}\cdot\bigstatevec{f}^{a,\text{GLM}} \right)= 0.
\end{equation}
The damping source term for the GLM divergence cleaning is zero in all but its ninth component, so we see
\begin{equation}
\statevec{w}^T\statevec{r} = -2\alpha\beta\psi^2.
\end{equation}
Thus, we have
\begin{equation}
 \iprod{S_t, 1} + \iprod{\vec{\nabla} \cdot \spacevec{f}^\ent, 1} = \iprod{\vec{\nabla} \cdot \bigstatevec{f}^v,\statevec{w}} + \iprod{-2\alpha \beta \psi^2, 1}.
\end{equation} 
Next, to demonstrate the entropy stability for the resistive GLM-MHD equations, we address the viscous flux components. So, we integrate by parts for the remaining inner product including the viscous fluxes and obtain
\begin{equation}
 \iprod{S_t, 1} + \iprod{\vec{\nabla} \cdot \spacevec{f}^\ent, 1} = \int\limits_{\partial \Omega} \statevec{w}^T\left(\bigstatevec{f}^v \cdot \spacevec{n} \right)\dS - \iprod{\bigstatevec{f}^v,\vec{\nabla}\statevec{w}} - \iprod{2\alpha \beta \psi^2, 1},
\end{equation} 
where $\spacevec{n}$ denotes the outward pointing normal vector at the domain boundaries. Assuming these to be periodic and using Lemma \ref{lemma1}, we arrive at
\begin{equation}
\iprod{S_t, t} + \iprod{\vec{\nabla} \cdot \spacevec{f}^\ent, 1} = - \iprod{\twentysevenMatrix{K}\vec{\nabla}\statevec{w}, \vec{\nabla}\statevec{w}} - \iprod{ 2\alpha \beta \psi^2, 1}.
\end{equation} 
Since $\twentysevenMatrix{K}$ is symmetric and positive semi-definite and $\alpha,\beta\geq 0$, all inner products on the right hand side are positive and thus
\begin{equation}
\iprod{S_t, 1} + \iprod{\vec{\nabla} \cdot \spacevec{f}^\ent, 1} \leq 0.
\end{equation} 
\end{proof}

\begin{rem}
Splitting the flux contraction in the continuous entropy analysis is useful to keep track of which terms contribute to the entropy flux and which terms cancel. This will also be the case in the discrete entropy analysis of the high-order DG approximation. As such, the entropy flux potential is split into Euler, ideal MHD and GLM components
\begin{equation}\label{eq:entPotential2}
\spacevec{\Psi} = \spacevec{\Psi}^{\text{Euler}} + \spacevec{\Psi}^{\text{MHD}} + \spacevec{\Psi}^{\text{GLM}},
\end{equation}
where
\begin{align}
\spacevec{\Psi}^{\text{Euler}} &= \statevec{w}^T\bigstatevec{f}^{a,\text{Euler}} - \spacevec{f}^\ent,\label{eq:EulerEntFluxPot}\\[0.125cm]
\spacevec{\Psi}^{\text{MHD}} &= \statevec{w}^T\bigstatevec{f}^{a,\text{MHD}} + \wDotJan \spacevec{B},\label{eq:MHDEntFluxPot}\\[0.125cm]
\spacevec{\Psi}^{\text{GLM}} &= \statevec{w}^T\bigstatevec{f}^{a,\text{GLM}}.\label{eq:GLMEntFluxPot}
\end{align}
\end{rem}

In summary, we have demonstrated that the resistive GLM-MHD equations satisfy an entropy inequality. To do so, we separated the advective contributions into Euler, ideal MHD and GLM pieces and considered the viscous contributions separately, which served to clarify how each term contributed to the entropy analysis. A major result is that it is possible to rewrite the resistive terms of the three-dimensional system in an entropy consistent way to demonstrate that those terms are entropy dissipative. We will use an identical splitting of the advective and diffusive terms in the discrete entropy stability proofs in Sec. \ref{Sec:ESDGres} to directly mimic the continuous analysis. However, we first must build the necessary components of a discontinuous Galerkin approximation for the resistive GLM-MHD equations.

\section{Split form discontinous Galerkin approximation}\label{Sec:DG}

The goal of this paper is to construct an entropy stable high-order DG method for the resistive GLM-MHD system \eqref{resGLMMHD}. 

In this section, we introduce the building blocks of our nodal DG discretization. Most importantly we highlight that, as long as the nodal DG approximation is built with the Legendre-Gauss-Lobatto nodes, the discrete derivative matrix and the discrete mass matrix satisfy the summation-by-parts (SBP) property for any polynomial order \cite{gassner_skew_burgers}. This is a key property as it allows us to use results from the work of Fisher et al. \cite{fisher2013_2} and Fisher and Carpenter \cite{fisher2013}. We follow the notation introduced in \cite{Gassner:2016ye} and present a split form DG approximation, where we have two numerical fluxes, one at the surface and one inside the volume for the special split form volume integral. Carpenter and Fisher showed that when using an entropy conservative finite volume flux for the numerical volume flux in a SBP discretization, the property of entropy conservation extends to the high-order SBP method, i.e. to the higher order DG method.

In the following we introduce the so-called nodal discontinuous Galerkin collocation spectral element method (DGSEM). We use tensor product elements and derive the strong form including the SBP property of the volume discretization. Furthermore, we restrict the DGSEM discretization to a two-dimensional Cartesian mesh with uniform squared elements. So, the first step in the discretization is to subdivide the computational domain $\Omega\subset\mathbb{R}^2$ into non-overlapping square elements $E_1,\ldots,E_K$ and map each of these elements to the reference element $E_0 = \left[-1,1\right]^2$. The bilinear mappings $X_k:E_0 \rightarrow E_k, X_k(\xi,\eta) = (x,y)$ are defined as
\begin{equation}
 X_k(\xi,\eta) = \frac{1}{4} \left[\spacevec{x}_1(1-\xi)(1-\eta)+\spacevec{x}_2(1+\xi)(1-\eta)+\spacevec{x}_3(1+\xi)(1+\eta)+\spacevec{x}_4(1-\xi)(1+\eta)\right]
\label{mapping}
\end{equation}
for $k=1,\ldots,K$, where $\{\spacevec{x}_1,\spacevec{x}_2,\spacevec{x}_3,\spacevec{x}_4\}$ are the four corners of the element $E_k$. Since we consider the uniform Cartesian case, we find the Jacobian of the mappings to be constant
\begin{equation}
J = \frac{1}{4} \Delta x \Delta y = \frac{\Delta x^2}{4}
\label{Jacobian}
\end{equation}
for equal element side lengths $\Delta x = \Delta x_k = \Delta y_k = \Delta y$, $k=1,\ldots, K$.

Next, on every element, each conserved variable is approximated by a polynomial of degree $N$ in each spatial direction on $E_0$
\begin{equation}
\statevec{u}(x,y,t)|_{E_k} = \statevec{u}(\xi,\eta,t) \approx \sum_{i,j = 0}^N \statevec{U}_{ij} \ell_i(\xi) \ell_j(\eta)\equiv \statevec{U},
\label{DG-approx}
\end{equation} 
where $\statevec{U}_{ij} \approx \statevec{u}(x,y,t)|_{E_k}$ are the time dependent nodal degrees of freedom for the considered element $E_k$ and time $t$. The interpolating Lagrange basis functions are defined by
\begin{equation}
\ell_i(\xi) = \prod_{\stackrel{j = 0}{j \neq i}}^N \frac{\xi-\xi_j}{\xi_i-\xi_j}\quad \text{ for } i=0,\ldots,N.
\label{lagrange}
\end{equation}
These basis functions are discretely orthogonal and satisfy the Kronecker-delta property, i.e. $\ell_j(\xi_i) = \delta_{ij}$ with $\delta_{ij} = 1$ for $i=j$ and $\delta_{ij}=0$ for $i\neq j$. To perform the discrete differentiation for polynomials of degree $N$, exactly, we introduce the polynomial derivative matrix by
\begin{equation}
\dmat_{ij} = \frac{\partial \ell_j}{\partial \xi} \Bigg{|}_{\xi = \xi_i} \text{ for } i,j=0,\ldots,N.
\label{Dmatrix}
\end{equation}
The nodal DGSEM is built from the weak formulation of the governing equations where any integrals in the inner products are approximated with a high-order Legendre-Gauss-Lobatto (LGL) quadrature
\begin{equation}\label{eq:quad}
\iprod{\statevec{f},\statevec{g}} = \int\limits_{E_0}\statevec{g}^T\statevec{f}\,\text{d}\xi\text{d}\eta \approx \sum\limits_{i,j=0}^N\statevec G^T(\xi_i,\eta_j)\statevec F(\xi_i,\eta_j)\omega_i\omega_j\equiv\iprodN{\statevec F,\statevec G},
\end{equation}
where $\left\{\xi_i\right\}_{i=0}^N, \left\{\eta_j\right\}_{j=0}^N$ are the LGL quadrature nodes and $\left\{\omega_i\right\}_{i=0}^N,\left\{\omega_j\right\}_{j=0}^N$ are the LGL quadrature weights. The choice of LGL quadrature is necessary, because then the discrete derivative matrix and the discrete mass matrix satisfy the SBP property for any polynomial order \cite{gassner_skew_burgers}
\begin{equation}
\mmat \dmat + (\mmat \dmat)^T = \qmat+\qmat^T = \bmat = \text{diag} (-1,0,\ldots,0,1).
\label{SBP}
\end{equation}
Here, the collocation of interpolation and quadrature nodes in the approximation introduces the discrete diagonal mass matrix
\begin{equation}
\mmat=\text{diag} (\omega_0, \ldots , \omega_N).
\label{Mmatrix}
\end{equation}
We also define the SBP matrix $\qmat$ and the boundary matrix $\bmat$ in \eqref{SBP}.

Starting from the original resistive GLM-MHD system we now have all ingredients to formulate the standard DGSEM approximation. We first multiply by a test function, $\testfuncOne$, and integrate over the domain $\Omega$
\begin{equation}
 \iprod{\statevec{u}_t, \testfuncOne} + \iprod{\vec{\nabla} \cdot \bigstatevec{f}^a(\statevec{u}), \testfuncOne} - \iprod{\vec{\nabla} \cdot \bigstatevec{f}^v(\statevec{u},\vec{\nabla} \statevec{u}), \testfuncOne}  +  \iprod{\JanC(\vec{\nabla}\cdot \spacevec{B}),\testfuncOne} = \iprod{\statevec{r},\testfuncOne}.
\end{equation} 
The viscous fluxes require the gradients of the solution. Thus, we introduce an auxiliary variable for the gradient of the entropy variables, since we know from the continuous analysis, that we can also express the viscous fluxes by the entropy variable gradients:  
\begin{equation}
\bigstatevec{q} = \vec{\nabla}\statevec{w}
\end{equation}
Now we can finally convert the governing equations into a first order system, condense the notation on the fluxes, multiply the evolution of the gradients by a different test function, $\testfuncTwo$, and integrate over $\Omega$ to obtain
\begin{equation}\label{eq:weakFormCont}
\begin{split}
\iprod{\statevec{u}_t, \testfuncOne} + \iprod{\vec{\nabla} \cdot \bigstatevec{f}^a, \testfuncOne} - \iprod{\vec{\nabla} \cdot \bigstatevec{f}^v, \testfuncOne}  +  \iprod{\JanC(\vec{\nabla}\cdot \spacevec{B}),\testfuncOne} & = \iprod{\statevec{r},\testfuncOne} \\
\iprod{\bigstatevec{q},\testfuncTwo} - \iprod{\vec{\nabla} \statevec{w},\testfuncTwo} & = 0
\end{split}
\end{equation} 

The discretization of the system is done by mapping into the reference element with \eqref{mapping} and replacing all quantities, e.g. the solution $\statevec{u}$, by polynomial interpolations in terms of \eqref{DG-approx}. Also, the integrals in the weak form \eqref{eq:weakFormCont} are approximated by the LGL quadrature rule \eqref{eq:quad}: 
\begin{equation}\label{eq:firstWeak}
\begin{split}
 \iprodN{J\statevec{U}_t, \testfuncOne} + \iprodN{\vec{\nabla}_\xi \cdot \bigcontravec{F}^a, \testfuncOne} - \iprodN{\vec{\nabla}_\xi \cdot \bigcontravec{F}^v, \testfuncOne} +  \iprodN{\JanD(\vec{\nabla}_\xi \cdot \contraspacevec{B}),\testfuncOne} & = \iprod{J\statevec{R},\testfuncOne} \\
\iprodN{J \bigstatevec{Q},\testfuncTwo} - \iprodN{\vec{\nabla}_\xi \contravec{\statevec{W}},\testfuncTwo} & = 0
\end{split}
\end{equation}
Since we restrict to uniform Cartesian meshes in two dimensions, the discretized, mapped variables in the collocated DGSEM read:
\begin{equation}
\bigcontravec{F}^a \approx \frac{\Delta x}{2}\bigstatevec{f}^a(\statevec{U}) ~~~~~~~ \bigcontravec{F}^v \approx \frac{\Delta x}{2}\bigstatevec{f}^v(\statevec{U},\bigstatevec{Q}) ~~~~~~~ \contrastatevec{W}\approx \frac{\Delta x}{2}\statevec{w} ~~~~~~~ \contraspacevec{B}\approx \frac{\Delta x}{2}\spacevec{B}
\end{equation}

From the SBP property \eqref{SBP} we have the relation
\begin{equation}\label{eq:otherSBP}
\dmat = \mmat^{-1}\bmat - \mmat^{-1}\dmat^T\mmat,
\end{equation}
where we use the fact that the mass matrix $\mmat$ is diagonal. Rewriting the polynomial derivative matrix by \eqref{eq:otherSBP} allows us to move discrete derivatives off the fluxes and the non-conservative term onto the test functions in \eqref{eq:firstWeak}. This process creates surface and volume contributions in the DGSEM. However, we must resolve the discontinuities in the surface approximation across element interfaces which naturally arise in DG methods. To do so, we introduce numerical flux functions $\bigstatevec{F}^{a,\ast},\bigcontravec{F}^{v,\ast},\contrastatevec{W}^\ast$. Additionally, for the non-conservative term we define the surface evaluation to be $\Bstar$ and obtain the discrete weak form:
\begin{equation}
\resizebox{0.925\hsize}{!}{$
\begin{split}
\iprodN{J\statevec{U}_t, \testfuncOne} & = - \int\limits_{\partial E,N} \testfuncOne^T \left[\bigcontravec{F}^{a,\ast}-\bigcontravec{F}^{v,\ast}\right] \cdot \spacevec{n} \dS - \int\limits_{\partial E,N} \testfuncOne^T \Bstar \cdot \spacevec{n} \dS + \left\langle \bigcontravec{F}^a, \vec{\nabla}_\xi \testfuncOne\right\rangle_N - \iprodN{\bigcontravec{F}^v, \vec{\nabla}_\xi \testfuncOne} \\[0.1cm]
& ~~~ +  \iprodN{\contraspacevec{B},\vec{\nabla}_\xi \left(\testfuncOne^T\JanD\right)} + \iprodN{J \statevec{R},\testfuncOne} \\[0.1cm]
\iprodN{J \bigstatevec{Q},\testfuncTwo} & = \int\limits_{\partial E,N} \testfuncTwo^T \contrastatevec{W}^\ast \cdot \spacevec{n} \dS -\iprodN{\contrastatevec{W},\vec{\nabla}_\xi\cdot \testfuncTwo}
\end{split}$}
\end{equation}  
The particular choices for the surface discretization of the non-conservative term and the numerical interface fluxes will be specified in the next section. Moreover, we introduce a compact notation for the surface terms
\begin{equation}
 \label{eq:discrete_surfint}
 \int\limits_{\partial E,N} {\left(\bigcontravec F \cdot \spacevec n\right)\,\dS} 
 = \sum\limits_{j = 0}^N {\left. {{\omega_{j}}{\contrastatevec F_{1}}\left( {\xi    ,{\eta_j}} \right)} \right|_{\xi   =  - 1}^1}  
  + \sum\limits_{i = 0}^N {\left. {{\omega_{i}}{\contrastatevec F_{2}}\left( {{\xi_i},\eta    } \right)} \right|_{\eta  =  - 1}^1}  
\equiv\int\limits_N {\left. {{\contrastatevec F_{1}}d\eta } \right|} _{\xi    =  - 1}^1 
       + \int\limits_N {\left. {{\contrastatevec F_{2}}d\xi  } \right|} _{\eta   =  - 1}^1 .
\end{equation}

As a final step we apply the SBP property once more to the advective flux and non-conservative term in the first equation and obtain the discrete strong form of the DGSEM: 
\begin{equation}
\resizebox{0.925\hsize}{!}{$
\begin{split}
 \iprodN{J\statevec{U}_t, \testfuncOne} & = - \int\limits_{\partial E,N} \testfuncOne^T \left[\bigcontravec{F}^{a,\ast}-\bigcontravec{F}^a \right]\cdot \spacevec{n} \dS - \iprodN{\vec{\nabla}_\xi \cdot \bigcontravec{F}^a, \testfuncOne} + \int\limits_{\partial E,N} \testfuncOne^T \left[\bigcontravec{F}^{v,\ast}-\bigcontravec{F}^v \right]\cdot \spacevec{n} \dS + \iprodN{\vec{\nabla}_\xi \cdot \bigcontravec{F}^v, \testfuncOne} \\[0.1cm]
& ~~~ - \int\limits_{\partial E,N} \testfuncOne^T \left[\Bstar - \JanD\contraspacevec{B}\right]\cdot \spacevec{n} \dS - \iprodN{\JanD(\vec{\nabla}_\xi\cdot \contraspacevec{B}),\testfuncOne} + \iprodN{J \statevec{R},\testfuncOne} \\[0.1cm]
\iprodN{J \bigstatevec{Q},\testfuncTwo} & = \int\limits_{\partial E,N} \testfuncTwo^T \contrastatevec{W}^\ast \cdot \spacevec{n} \dS - \iprodN{\contrastatevec{W},\vec{\nabla}_\xi \cdot\testfuncTwo} 
\end{split}$}
\end{equation}

Before we can build a DGSEM approximation that is entropy stable for the resistive GLM-MHD system, we approximate the volume integral contribution of the advective flux terms in a split form fashion \cite{fisher2013,Gassner:2016ye} by
\begin{equation}
\label{eq:entropy-cons_volint}
\left\{\spacevec\nabla _\xi \cdot \bigcontravec{F}^a\right\}_{ij} = \left\{\spacevec{\mathbb{D}}\bigcontravec{F}^{a,\#}\right\}_{ij} = 2\sum_{m=0}^N \dmat_{im}\contrastatevec{F}_1^{a,\#}(\statevec{U}_{ij}, \statevec{U}_{mj}) + 2\sum_{m=0}^N \dmat_{jm} \contrastatevec{F}_2^{a,\#}(\statevec{U}_{ij}, \statevec{U}_{im})
\end{equation}
for $i,j=0,\ldots N$, where we introduce the two-point, symmetric volume flux $\bigcontravec{F}^{a,\#}$, which is also specified in the next section. The split formulation of the DG approximation \eqref{eq:entropy-cons_volint} offers the flexibility in the DGSEM to satisfy auxiliary properties such as kinetic energy preservation \cite{kennedy2008,jameson2008,pirozzoli2011}, entropy stability \cite{carpenter_esdg,Liu2017} or both \cite{Ray2017,Gassner:2016ye}. Kinetic energy preservation has proven to be of interest in the realm of turbulence modeling, e.g. \cite{flad2017}. Split forms also offer enhanced robustness for under-resolved computations, e.g. for the compressible Euler equations \cite{Gassner:2016ye}, because they include a built-in dealiasing mechanism for non-linear problems while remaining conservative for approximations that satisfy the SBP property, e.g. \cite{fisher2013,carpenter_esdg,Gassner:2016ye}. In this work we restrict the discussion to a provably entropy stable DGSEM for the resistive GLM-MHD equations as there is a known equivalence between split forms and entropy stability \cite{tadmor1984}.

In contrast, the volume contributions of the non-conservative terms and the viscous fluxes in the resistive GLM-MHD equations are approximated by the standard DG derivatives
\begin{equation}\label{eq:nonConsDerivative}
\begin{split}
\left\{\JanD(\spacevec\nabla_\xi\cdot\contraspacevec{B})\right\}_{ij} & = \left\{\JanD \spacevec{\mathbb{D}}^S\contraspacevec{B}\right\}_{ij}= \sum_{m=0}^N \dmat_{im}\,\JanD_{ij} \left(\tilde{B}_1\right)_{mj} + \sum_{m=0}^N \dmat_{jm}\,\JanD_{ij} \left(\tilde{B}_2\right)_{im} \\
\left\{\spacevec\nabla _\xi \cdot \bigcontravec F^v\right\}_{ij} & = \left\{\spacevec{\mathbb{D}}^S\bigcontravec{F}^{v}\right\}_{ij} = \sum_{m=0}^N \dmat_{im}\,\left(\contrastatevec{F}^{v}_1\right)_{mj} + \sum_{m=0}^N \dmat_{jm}\,\left(\contrastatevec{F}^{v}_2\right)_{im}
\end{split}
\end{equation}
for $i,j=0,\ldots N$.

We finally obtain the baseline DGSEM approximation, needed in the next section to prove entropy stability for the resistive GLM-MHD equations, by the divergence approximations  \eqref{eq:entropy-cons_volint} and \eqref{eq:nonConsDerivative}, the yet arbitrary surface and volume numerical fluxes and the additional magnetic fields at the element boundaries to give
\begin{equation}\label{eq:schemeFinal}
\resizebox{0.925\hsize}{!}{$
\begin{split}
 \iprodN{J\statevec{U}_t, \testfuncOne} & = - \int\limits_{\partial E,N} \testfuncOne^T \left[\bigcontravec{F}^{a,\ast}-\bigcontravec{F}^a \right]\cdot \spacevec{n} \dS - \iprodN{\spacevec{\mathbb{D}} \bigcontravec{F}^{a,\#}, \testfuncOne} + \int\limits_{\partial E,N} \testfuncOne^T \left[\bigcontravec{F}^{v,\ast}-\bigcontravec{F}^v \right]\cdot \spacevec{n} \dS + \iprodN{\mathbb{D}^S \bigcontravec{F}^v, \testfuncOne} \\[0.1cm]
& ~~~ - \int\limits_{\partial E,N} \testfuncOne^T \left[\Bstar - \JanD\contraspacevec{B}\right]\cdot \spacevec{n} \dS - \iprodN{\JanD \spacevec{\mathbb{D}}^S \contraspacevec{B},\testfuncOne} + \iprodN{J \statevec{R},\testfuncOne} \\[0.1cm]
\iprodN{J \bigstatevec{Q},\testfuncTwo} & = \int\limits_{\partial E,N} \testfuncTwo^T \contrastatevec{W}^\ast \cdot \spacevec{n} \dS - \iprodN{\contrastatevec{W},\vec{\nabla}_\xi \cdot \testfuncTwo}
\end{split}$}
\end{equation}  

The resulting ordinary differential equations are integrated in time by an explicit 4th order low storage Runge-Kutta method of Carpenter and Kennedy \cite{Carpenter&Kennedy:1994} for each element $k=1,\ldots,K$. One possible choice for a stable explicit time step reads
\begin{equation}
\Delta t = \min\{\Delta t^a,\Delta t^v\},
\end{equation} 
where we select the advective time step using the Courant, Friedrichs and Lewy (CFL) condition, which for DG type approximations in two spatial dimensions on square Cartesian meshes is \cite{cockburn2001,gassner2011}
\begin{equation}\label{eq:CFL_timestep_2D}
\Delta t^a \le \frac{\mathtt{CFL}}{|\lambda^a_\mathrm{max}|}\left(\frac{\Delta x}{2N+1}\right).
\end{equation}
Here $\lambda^a_\mathrm{max}$ is the largest advective wave speed at the current time traveling in either the $\{x,y\}$ direction, $N$ is the polynomial order of the approximation and $\mathtt{CFL}\in(0,1]$ is an adjustable coefficient. A similar condition is used for the viscous time step selection \cite{gassner2011}
\begin{equation}\label{eq:CFL_timestep_2D_Visc}
\Delta t^v \le \frac{\mathtt{DFL}}{|\lambda^v_\mathrm{max}|}\left(\frac{\Delta x}{2N+1}\right)^2,
\end{equation}
where $\lambda^v_\mathrm{max}$ is the largest eigenvalue of the viscous flux Jacobian in either the $\{x,y\}$ direction and $\mathtt{DFL}\in(0,1]$ is a conventional coefficient. Full details on the advective and viscous time step selection can be found in \cite{Derigs2017,altmann2012,dumbser2009}.

Finally, it is still open how to select the numerical surface, volume and viscous fluxes as well as the evaluation of the surface non-conservative term $\Bstar$ in \eqref{eq:schemeFinal} to obtain entropy stability. The choice of numerical surface and volume fluxes is based on previous work from Derigs et al. \cite{Derigs2017} to obtain entropy stability for the ideal GLM-MHD equations and Carpenter et al. \cite{carpenter_esdg} to retain high-order accuracy in the DGSEM. However, we highlight the contributions of the GLM terms in the DG context and how they affect the discrete entropy in the next sections. We note that Liu et al. \cite{Liu2017} directly presented the final form of the non-conservative surface term. We, however, provide a motivation for the non-conservative surface discretization in the context of the split form DG framework. We demonstrate in this work that there is a proper choice to evaluate the coupling term for the non-conservative term at element interfaces such that the term of Liu et al. is recovered and the entropy analysis holds. Further, special attention is given to selecting the viscous numerical fluxes and proving entropy stability for the discrete resistive terms, which expands on the recent results of Gassner et al. for the compressible Navier-Stokes equations \cite{Gassner2017}.

\section{Entropy stable DG scheme for resistive GLM-MHD}\label{Sec:ESDGres}

Much work in the numerics community has been invested over the years to develop approximations of non-linear hyperbolic PDE systems that remain thermodynamically consistent, e.g. \cite{Chandrashekar2015,Fjordholm2011,gassner_skew_burgers,tadmor2016,wintermeyer2017,Winters2016}. This began with the pioneering work of Tadmor \cite{tadmor1984,Tadmor1987_2} to develop low-order finite volume approximations. Extension to higher spatial order was recently achieved in the context of DG methods for the compressible Navier-Stokes equations \cite{carpenter_esdg,Gassner:2016ye} as well as the ideal MHD equations \cite{Liu2017,Valencia2017}. Remarkably, Carpenter and Fisher et al. showed that the conditions to develop entropy stable approximations at low-order immediately apply to high-order methods provided the derivative approximation satisfies the SBP property \cite{carpenter_esdg,fisher2013,fisher2013_2}. 

In order to get entropy stability we start with the derived split form DG approximation \eqref{eq:schemeFinal} and contract into entropy space by replacing the first test function with the interpolant of the entropy variables and the second one with the interpolant of the viscous fluxes to obtain:
\begin{equation}\label{eq:schemeFinalcontr}
\begin{aligned}
\iprodN{J\statevec{U}_t, \statevec{W}} & = - \iprodN{\spacevec{\mathbb{D}} \bigcontravec{F}^{a,\#}, \statevec{W}} - \iprodN{\JanD \spacevec{\mathbb{D}}^S \contraspacevec{B}, \statevec{W}} + \iprodN{J \statevec{R}, \statevec{W}} \\[0.125cm]
& ~~~  - \int\limits_{\partial E,N} \statevec{W}^T \left[\bigcontravec{F}^{a,\ast}-\bigcontravec{F}^a \right]\cdot \spacevec{n} \dS - \int\limits_{\partial E,N}  \statevec{W}^T \left[\Bstar - \JanD\contraspacevec{B}\right]\cdot \spacevec{n} \dS  \\[0.125cm]
& ~~~ + \iprodN{\mathbb{D}^S \bigcontravec{F}^v, \statevec{W}} + \int\limits_{\partial E,N} \statevec{W}^T \left[\bigcontravec{F}^{v,\ast}-\bigcontravec{F}^v \right]\cdot \spacevec{n} \dS \\[0.125cm]
\iprodN{J \bigstatevec{Q},\bigcontravec{F}^v} & = \int\limits_{\partial E,N} \left(\bigcontravec{F}^v\right)^T \statevec{W}^\ast \cdot \spacevec{n} \dS - \iprodN{\statevec{W},\mathbb{D}^S \bigcontravec{F}^v} 
\end{aligned}
\end{equation}
Here, we have intentionally arranged the advective plus non-conservative volume parts, the advective plus non-conservative surface parts and the viscous parts of the first equation into separate rows.

The time derivative term in \eqref{eq:schemeFinalcontr} is the time rate of change of the entropy in the element. Assuming that the chain rule with respect to differentiation in time holds (time continuity), we use the contraction property of the entropy variable \eqref{eq:wsContraction} at each LGL node within the element to see that on each element $k=1,\ldots,K$ we have
\begin{equation}\label{totalEntr}
\iprodN{J\statevec{U}_t,\statevec{W}} = J\!\!\sum\limits_{i,j=0}^N \omega_i \omega_j \statevec{W}^T_{ij}\frac{d \statevec{U}_{ij}}{d t} = J\!\!\sum\limits_{i,j=0}^N \omega_i \omega_j \frac{d S_{ij}}{d t} = \iprodN{JS_t,1}.
\end{equation}
To obtain the total discrete entropy we sum over all elements
\begin{equation}
\frac{d \overline{S}}{d t}\equiv\sum\limits_{k=1}^K\iprodN{J^kS^k_t,1}.
\end{equation}
The final goal of this section is to demonstrate the entropy stability of the contracted DG approximation \eqref{eq:schemeFinalcontr} for the resistive GLM-MHD system, i.e.,
\begin{equation}\label{eq:whatWeWant}
\frac{d \overline{S}}{d t} \leq 0.
\end{equation}

To build the result \eqref{eq:whatWeWant} we examine each row in the first equation of \eqref{eq:schemeFinalcontr} incrementally. In Sec. \ref{Sec:AdvParts}, we demonstrate the behavior of the advective and non-conservative volume as well as interface contributions. Throughout this section we highlight how the GLM divergence cleaning terms affect the approximation. Then, in Sec. \ref{Sec:ResParts}, we account for the viscous components of the scheme and demonstrate the second main result of this work. In particular, we find that the entropy stability of the scheme \eqref{eq:schemeFinalcontr} depends on the ability to rewrite the viscous fluxes in terms of the gradient of the entropy variables \eqref{Kgradient}. Also, the selection of the viscous numerical flux functions must take the entropy variables into account, as was previously noted for the compressible Navier-Stokes equations by Gassner et al. \cite{Gassner2017}. 

\subsection{Analysis of the advective parts}\label{Sec:AdvParts}

This section focuses on the advective parts in the contracted DG approximation \eqref{eq:schemeFinalcontr}. First, we select the specific form of the advective interface and volume numerical fluxes in section \ref{Sec:Fluxes}. In the next section, \ref{Sec:Vol}, we show that the volume contributions of the entropy conservative flux of the Euler terms become the entropy flux at the surfaces, the ideal MHD and non-conservative terms cancel and the GLM terms vanish. By splitting the entropy conservative flux into three terms we explicitly see how the discrete contraction into entropy space mimics the results of the continuous analysis, i.e., \eqref{eq:fluxContractionEuler}, \eqref{eq:fluxContractionMHD} and \eqref{eq:fluxContractionGLM}. Next, with the knowledge that the volume contributions move to the interfaces, Sec. \ref{Sec:Surf} addresses all the surface contributions and we select the form of the coupling for the non-conservative term. By summing over all the elements and applying the definition of the entropy conservative fluxes we cancel all the remaining advective and non-conservative terms. 

\subsubsection{Numerical entropy fluxes}\label{Sec:Fluxes}

A consistent, symmetric numerical flux function, which is entropy conservative for the ideal GLM-MHD equations, is derived in the finite volume context \cite{Derigs2017} and serves as the backbone for the high-order entropy stable DGSEM considered in this work. First, we define the notation for the jump operator, arithmetic and logarithmic means between a left and right state, $a_L$ and $a_R$, respectively
\begin{equation}
\jump{a} := a_R-a_L, ~~~~~~~~~ \avg{a} := \frac{1}{2}(a_L+a_R), ~~~~~~~~ a^{\ln} := \jump{a}/\jump{\ln(a)},
\label{means}
\end{equation}
where a numerically stable procedure to evaluate the logarithmic mean is given in \cite{IsmailRoe2009}. We present the entropy conserving (EC) numerical flux in the first spatial direction to be
\begin{equation}\label{ECFlux}
\statevec{f}^{\ec} = \begin{pmatrix} \rho^{\ln} \avg{v_1} \\ \rho^{\ln} \avg{v_1}^2 - \avg{B_1}^2 + \overline{p} + \frac{1}{2} \Big(\avg{B_1 B_1} + \avg{B_2 B_2} + \avg{B_3 B_3}\Big)\\ \rho^{\ln} \avg{v_1} \avg{v_2} - \avg{B_1} \avg{B_2} \\ \rho^{\ln} \avg{v_1} \avg{v_3} - \avg{B_1} \avg{B_3} \\ f_5^{\ec} \\ c_h \avg{\psi} \\ \avg{v_1}\avg{B_2} - \avg{v_2}\avg{B_1}\\ \avg{v_1}\avg{B_3} - \avg{v_3}\avg{B_1}\\c_h \avg{B_1} \end{pmatrix}
\end{equation}
with 
\begin{equation}
\begin{split}
f_5^{\ec} = & f_1^{\ec}\bigg[\frac{1}{2 (\gamma-1) \beta^{\ln}} - \frac{1}{2} \left(\avg{v_1^2} + \avg{v_2^2} + \avg{v_3^2}\right) \bigg] + f_2^{\ec} \avg{v_1} + f_3^{\ec} \avg{v_2} + f_4^{\ec} \avg{v_3} \\
 & + f_6^{\ec} \avg{B_1} + f_7^{\ec} \avg{B_2} + f_8^{\ec} \avg{B_3} + f_9^{\ec} \avg{\psi} - \frac{1}{2} \big(\avg{v_1 B_1^2}+\avg{v_1 B_2^2}+\avg{v_1 B_3^2}\big)\\
 & + \avg{v_1 B_1} \avg{B_1}+\avg{v_2 B_2} \avg{B_1}+\avg{v_3 B_3} \avg{B_1} - c_h \avg{B_1 \psi}
\end{split}
\end{equation}
and
\begin{equation*}
\overline{p} = \frac{\avg{\rho}}{2\avg{\beta}}.
\end{equation*}
This particular choice for $\statevec{f}^*$ satisfies the discrete entropy conservation condition \cite{Chandrashekar2015,Derigs2017,Liu2017,Valencia2017}
\begin{equation}\label{discECcond}
\jump{\statevec{w}}^T\statevec{f}^\ec = \jump{\Psi_1} - \avg{B_1}\jump{\wDotJan}
\end{equation}
with the entropy flux potential $\spacevec{\Psi}$ \eqref{eq:entPotential} and the contracted non-conservative state vector $\wDotJan$ \eqref{eq:contractSource}.
 
Moreover, in the presence of shocks or discontinuities, we must add dissipation to the interface fluxes in terms of the entropy variables  to ensure we do not violate the entropy inequality \eqref{entrineq}. In order to create such an entropy stable scheme, we use the EC flux in \eqref{ECFlux} as a baseline flux and add a general form of numerical dissipation at the interfaces to get an entropy stable (ES) numerical flux that is applicable to arbitrary flows
\begin{equation}\label{ES_flux}
	\statevec{f}^{\es} =  \statevec{f}^{\ec} - \frac{1}{2} \nineMatrix{\Lambda} \nineMatrix{H} \jump{\statevec{w}},
\end{equation}
where $\nineMatrix{H}$ is the entropy Jacobian and $\nineMatrix{\Lambda}$ the dissipation matrix. Full details can be found in \cite{Derigs2017}. It is straightforward to derive similar EC and ES fluxes for the $y-$direction.

In the contracted DG approximation \eqref{eq:schemeFinalcontr} we select both the two point volume fluxes $\bigcontravec{F}^{a,\#}$ and the advective surface fluxes $\bigcontravec{F}^{a,\ast}$ to be the EC fluxes
\begin{equation}\label{ECVolSurfluxes}
\bigcontravec{F}^{a,\#} = \bigcontravec{F}^\ec,\quad\bigcontravec{F}^{a,\ast} = \bigcontravec{F}^\ec,
\end{equation}
whereas the latter can also include stabilization terms as in \eqref{ES_flux}.

Again, as in the continuous analysis, we split the EC numerical flux function into three components
\begin{equation}\label{ECfluxsplit}
\bigcontravec{F}^\ec = \bigcontravec{F}^{\ec,\text{Euler}} + \bigcontravec{F}^{\ec,\text{MHD}} + \bigcontravec{F}^{\ec,\text{GLM}},
\end{equation}
as well as the appropriate entropy conservation conditions for the numerical flux functions,
\begin{align}
\jump{\statevec{w}}^T\bigcontravec{F}^{\ec,\text{Euler}} &= \jump{\spacevec{\Psi}^{\text{Euler}}}\label{eq:EulerPartDiscrete},\\[0.1cm]
\jump{\statevec{w}}^T\bigcontravec{F}^{\ec,\text{MHD}} &= \jump{\spacevec{\Psi}^{\text{MHD}}} - \avg{\contraspacevec{B}}\jump{\wDotJan}\label{eq:MHDPartDiscrete},\\[0.1cm]
\jump{\statevec{w}}^T\bigcontravec{F}^{\ec,\text{GLM}} &= \jump{\spacevec{\Psi}^{\text{GLM}}}\label{eq:GLMPartDiscrete},
\end{align}
where we use the previously defined split entropy flux potential \eqref{eq:entPotential2}. Respectively, \eqref{eq:EulerPartDiscrete} contains the hydrodynamic contributions, \eqref{eq:MHDPartDiscrete} contains the magnetic field parts and \eqref{eq:GLMPartDiscrete} contains the $\psi$ components. 

With the advective numerical fluxes now defined and the splitting of the discrete entropy conditions at hand we are prepared to examine the volume and surface discretizations in entropy space.

\subsubsection{Volume contributions}\label{Sec:Vol}

In this section we focus on the volume discretizations of the advective fluxes as well as the non-conservative terms in the first row of \eqref{eq:schemeFinalcontr}. We utilize the flux splitting \eqref{ECfluxsplit} to determine the contributions from the Euler, MHD and GLM parts, separately. Since the first two have been investigated in the DG context, see e.g. \cite{Liu2017,Valencia2017}, we address the novel GLM flux part first. 
\begin{lem}[Entropy contribution of GLM volume terms]\label{lem:GLM_EC_vol}~\newline
The GLM volume contributions of \eqref{eq:schemeFinal} reduce to zero in entropy space, i.e.,
\begin{equation}\label{eq:volInProof}
\iprodN{\spacevec{\mathbb{D}}\bigcontravec F^{\ec,\mathrm{GLM}},\statevec W} = 0.
\end{equation}
\end{lem}
\begin{proof}
We first expand the volume contribution from the GLM terms to be
\begin{equation}
\resizebox{0.925\hsize}{!}{$
\iprodN{\spacevec{\mathbb{D}}\bigcontravec F^{\ec,\text{GLM}},\statevec{W}} = 
 \sum\limits_{i,j=0}^{N}\omega_{i}\omega_j\statevec{W}^T_{ij}\left[\frac{\Delta x}{2}\sum\limits_{m=0}^N 2\dmat_{im}\statevec{F}_1^{\ec,\text{GLM}}(\statevec{U}_{ij}, \statevec{U}_{mj}) + \frac{\Delta x}{2}\sum\limits_{m=0}^N 2\dmat_{jm}\statevec{F}_2^{\ec,\text{GLM}}(\statevec{U}_{ij}, \statevec{U}_{im})\right],$}
\end{equation}
where we peeled the constant mapping term $\frac{\Delta x}{2}$ out of the entropy conservative fluxes for convenience.

We focus on the $\xi-$direction term of the volume integral approximations, as the $\eta-$direction is done in an analogous manner. The sum can be written in terms of the SBP matrix \eqref{SBP}, $\qmat_{im}=\omega_{i}\dmat_{im}$,
\begin{equation}
\resizebox{0.925\hsize}{!}{$
\frac{\Delta x}{2}\sum\limits_{j=0}^{N}\omega_{j}\sum\limits_{i=0}^N\statevec{W}^T_{ij}\sum\limits_{m=0}^N 2\omega_i\dmat_{im}\statevec{F}_1^{\ec,\text{GLM}}(\statevec{U}_{ij}, \statevec{U}_{mj})
=\frac{\Delta x}{2}\sum\limits_{j=0}^{N}\omega_{j}\sum\limits_{i=0}^N\statevec{W}^T_{ij}\sum\limits_{m=0}^N 2\qmat_{im}\statevec{F}_1^{\ec,\text{GLM}}(\statevec{U}_{ij}, \statevec{U}_{mj}).$}
\end{equation}
We use the summation-by-parts property $2\qmat_{im} = \qmat_{im} - \qmat_{mi} + \bmat_{im}$, perform a reindexing of $i$ and $m$ to incorperate the $\qmat_{mi}$ term and use the fact that $\statevec{F}_1^{\ec,\text{MHD}}(\statevec{U}_{ij}, \statevec{U}_{mj})$ is symmetric with respect to the index $i$ and $m$ to rewrite the $\xi-$direction contribution to the volume integral approximation as
\begin{equation}\label{eq:summ1}
\begin{aligned}
\sum\limits_{i=0}^N\statevec{W}^T_{ij}\sum_{m=0}^N 2\qmat_{im}\statevec{F}_1^{\ec,\text{GLM}}(\statevec{U}_{ij}, \statevec{U}_{mj})
&= \sum_{i,m=0}^N\statevec W^T_{ij}(\qmat_{im} - \qmat_{mi} + \bmat_{im})\statevec{F}_1^{\ec,\text{GLM}}(\statevec{U}_{ij},\statevec{U}_{mj})\\
&= \sum_{i,m=0}^N\qmat_{im}\left(\statevec W_{ij} - \statevec W_{mj}\right)^T
\statevec{F}_1^{\ec,\text{GLM}}(\statevec{U}_{ij},\statevec{U}_{mj}) \\
&\qquad\qquad\quad + \bmat_{im}\,\statevec W^T_{ij}\statevec{F}_1^{\ec,\text{GLM}}(\statevec{U}_{ij},\statevec{U}_{mj}).
\end{aligned}
\end{equation}
We have divided the entropy flux potentials into Euler, ideal MHD and GLM components \eqref{eq:EulerEntFluxPot}-\eqref{eq:GLMEntFluxPot}. Because the proof at hand only considers the GLM terms we are only concerned with the term
\begin{equation}
\Psi_1^{\text{GLM}}= \statevec{w}^T\statevec{f}_1^{\text{GLM}}
\end{equation}
and the accompanying entropy conservation condition \eqref{eq:GLMPartDiscrete}
\begin{equation}\label{eq:entConditionInProof}
\jump{\statevec{w}}^T\statevec{f}_1^{\ec,\text{GLM}} = \jump{\Psi_1^{\text{GLM}}}.
\end{equation}
We apply the form of \eqref{eq:entConditionInProof} to rewrite \eqref{eq:summ1} with
\begin{equation}\label{eq:moreBs}
\left(\statevec W_{ij} - \statevec W_{mj}\right)^T\statevec{F}_1^{\ec,\text{GLM}}(\statevec{U}_{ij},\statevec{U}_{mj}) = \left(\Psi^{\text{GLM}}_1\right)_{ij}-\left(\Psi^{\text{GLM}}_1\right)_{mj}.
\end{equation}
Furthermore, note that the entries of the boundary matrix $\bmat$ are only non-zero when $i=m=0$ or $i=m=N$, so
\begin{equation}\label{eq:boundaryBs}
\bmat_{im}\statevec{W}^T_{ij}\statevec{F}_1^{\ec,\text{GLM}}(\statevec{U}_{ij},\statevec{U}_{mj}) = \bmat_{im}\left(\Psi_1^{\text{GLM}}\right)_{ij}.
\end{equation}
We substitute \eqref{eq:moreBs} and \eqref{eq:boundaryBs} into the second line of \eqref{eq:summ1} to find
\begin{equation}\label{eq:awfulSum}
\begin{aligned}
\sum\limits_{i=0}^N\statevec W^T_{ij}\sum_{m=0}^N 2\qmat_{im}\statevec{F}_1^{\ec,\text{GLM}}(\statevec{U}_{ij}, \statevec{U}_{mj})
&=\sum\limits_{i,m=0}^N\qmat_{im}\left[\left({\Psi}_1^{\text{GLM}}\right)_{ij} - \left({\Psi}_1^{\text{GLM}}\right)_{mj}\right] + \bmat_{im}\left(\Psi_1^{\text{GLM}}\right)_{ij}.
\end{aligned}
\end{equation}
We will examine the terms of the sum \eqref{eq:awfulSum} moving from left to right. Now, because the derivative of a constant is zero (i.e. the rows of $\qmat$ sum to zero),
\begin{equation}\label{eq:B11}
\sum\limits_{i,m=0}^N\qmat_{im}\left(\Psi_1^{\text{GLM}}\right)_{ij} = \sum\limits_{i=0}^N\left(\Psi_1^{\text{GLM}}\right)_{ij}\sum\limits_{m=0}^N\qmat_{im} = 0.
\end{equation}
Next, on the second term, we use the SBP property \eqref{SBP} and reindex on the $\qmat_{mi}$ term to rewrite
\begin{equation}\label{eq:B12}
\begin{aligned}
-\!\!\!\sum\limits_{i,m=0}^N\qmat_{im}\left(\Psi_1^{\text{GLM}}\right)_{mj} 
&= -\!\!\!\sum\limits_{i,m=0}^N(\bmat_{im}-\qmat_{mi})\left(\Psi_1^{\text{GLM}}\right)_{mj} \\
&=-\!\!\!\sum\limits_{i,m=0}^N(\bmat_{im}-\qmat_{im})\left(\Psi_1^{\text{GLM}}\right)_{ij} \\
&=-\!\!\!\sum\limits_{i,m=0}^N\bmat_{im}\left(\Psi_1^{\text{GLM}}\right)_{ij} + \sum\limits_{i,m=0}^N\qmat_{im}\left(\Psi_1^{\text{GLM}}\right)_{ij} \\
&=-\!\!\!\sum\limits_{i,m=0}^N\bmat_{im}\left(\Psi_1^{\text{GLM}}\right)_{ij} + \sum\limits_{i=0}^N\left(\Psi_1^{\text{GLM}}\right)_{ij}\sum\limits_{m=0}^N \qmat_{im}\\
&=-\!\!\!\sum\limits_{i,m=0}^N\bmat_{im}\left(\Psi_1^{\text{GLM}}\right)_{ij} + 0,
\end{aligned}
\end{equation}
where, again, one term drops out due to consistency of the matrix $\qmat$.
Therefore, from \eqref{eq:B12}, we now have
\begin{equation}\label{eq:awfulSum3}
\begin{aligned}
\sum\limits_{i=0}^N\statevec W^T_{ij}\sum_{m=0}^N 2\qmat_{im}\statevec{F}_1^{\ec,\text{GLM}}(\statevec{U}_{ij}, \statevec{U}_{mj})
&=\sum\limits_{i,m=0}^N-\bmat_{im}\left({\Psi}_1^{\text{GLM}}\right)_{ij} + \bmat_{im}\left(\Psi_1^{\text{GLM}}\right)_{ij} = 0.
\end{aligned}
\end{equation}

An analogous result to \eqref{eq:awfulSum3} holds in the $\eta-$direction of the volume integral approximation, leading to the desired result
\begin{equation}
\iprodN{\spacevec{\mathbb{D}}\bigcontravec F^{\ec,\text{GLM}},\statevec W} = 0.
\end{equation}
\end{proof}

\begin{rem}\label{ES_damp}
If we also take the damping source term of the GLM divergence cleaning into account, the statement of Lemma \ref{lem:GLM_EC_vol} becomes an inequality, i.e.
\begin{equation}\label{eq:volInProofDamp}
\iprodN{\spacevec{\mathbb{D}}\bigcontravec F^{\ec,\mathrm{GLM}},\statevec W} + \left\langle J \statevec{R}, \statevec{W}\right\rangle_N \leq 0,
\end{equation}
since
\begin{equation}
 \left\langle J \statevec{R},\statevec{W} \right\rangle_N =  -J\sum\limits_{i,j=0}^N \omega_i \omega_j \left(2 \alpha \beta_{ij} \psi_{ij}^2 \right)\leq  0,
\end{equation}
for $\alpha,\beta_{ij} \geq  0$. This result corresponds to discrete entropy stability instead of conservation and will be excluded for the following discussion of the remaining advective parts.
\end{rem}

\begin{cor}[Entropy contribution of total volume terms]\label{cor:EC_vol}~\newline
For each element the sum of all advective volume contributions plus the non-conservative volume terms in \eqref{eq:schemeFinalcontr} yields
\begin{equation}
\iprodN{\spacevec{\mathbb{D}} \bigcontravec{F}^\ec, \statevec{W}} + \iprodN{\JanD \spacevec{\mathbb{D}}^S \contraspacevec{B},\statevec{W}} = \int\limits_{\partial E,N} \left(\contraspacevec{F}^\ent\cdot\spacevec n\right) \dS.
\end{equation}
\end{cor}
\begin{proof}
Again we first split up the volume flux in Euler, MHD and GLM parts according to \eqref{ECfluxsplit}. From Lemma \ref{lem:GLM_EC_vol} we know, that the GLM volume flux vanishes exactly. Moreover, we know from \cite{fisher2013,Gassner2017}, that the volume contributions of the Euler components become the entropy flux evaluated at the boundary 
\begin{equation}\label{eq:volSurfTerms}
\iprodN{\spacevec{\mathbb{D}}\bigcontravec F^{\ec,\text{Euler}},\statevec W} = \int\limits_{\partial E,N} \left(\contraspacevec{F}^\ent\cdot\spacevec n\right) \dS,
\end{equation}
which is equivalent to the step \eqref{eq:fluxContractionEuler} in the continuous analysis.

Finally, the remaining volume contributions of the ideal MHD equations and the non-conservative terms cancel in entropy space \cite{Liu2017} 
\begin{equation}\label{eq:MHDVolTerms}
\iprodN{\spacevec{\mathbb{D}}\bigcontravec{F}^{\ec,\text{MHD}},\statevec W} + \iprodN{\JanD\spacevec{\mathbb{D}}^S\contraspacevec{B},\statevec{W}} = 0,
\end{equation}
which is equivalent to the step \eqref{eq:fluxContractionMHD} in the continuous analysis. We provide a complete proof of this cancellation property in \ref{Sec:MHD_vol}, which expands upon the one dimensional proof of Liu et al. \cite{Liu2017}.
\end{proof}

The results of Lemma \ref{lem:GLM_EC_vol} and Corollary \ref{cor:EC_vol} demonstrate that many of the volume contributions cancel in entropy space and the remaining terms move to the interfaces of the contracted DG approximation. Thus, in the next section we include this additional interface contribution containing the entropy fluxes.

\subsubsection{Surface contributions}\label{Sec:Surf}

We are now prepared to examine the surface terms of the contracted DG approximation \eqref{eq:schemeFinalcontr} incorporating the now known additional surface part that comes from the volume terms due to the result of Corollary \ref{cor:EC_vol}. On each element the surface terms are given in compact notation as
\begin{equation}\label{eq:surfaceTermsOnK}
\Gamma_k = \int\limits_{\partial E_k,N} \statevec{W}^T \left[\bigcontravec{F}^{\ec}-\bigcontravec{F}^a \right]\cdot \spacevec{n} \dS + \int\limits_{\partial E_k,N}  \statevec{W}^T \left[\Bstar - \JanD\contraspacevec{B}\right]\cdot \spacevec{n} \dS +  \int\limits_{\partial E_k,N} \left(\contraspacevec{F}^\ent\cdot\spacevec n\right) \dS.
\end{equation}
To determine the total surface contributions from the advective and non-conservative terms in the contracted DG approximation \eqref{eq:schemeFinalcontr} we sum over all elements, $k=1,\ldots,K$ similar to Gassner et al. \cite{Gassner2017}. We introduce notation for states at the first LGL node of the neighboring element to be ``$\sl$'' and compliment the notation with ``$\ma$'' to denote the value of a quantity at the boundary LGL nodes on the current element. This allows us to define the orientated jump and the arithmetic mean at the interfaces to be
\begin{equation}\label{eq:jumpNotation}
\jump{\cdot} = (\cdot)^{\sl} - (\cdot)^{\ma},\quad \avg{\cdot} = \frac{1}{2}\left((\cdot)^{\sl}+(\cdot)^{\ma}\right).
\end{equation}
We note that, as we focus on periodic domains, the surface contributions at the physical domain boundaries can be expressed with the same notation as the interior edges containing periodic jumps and averages from opposite domain boundaries.

We investigate the total surface contributions from \eqref{eq:surfaceTermsOnK} term by term. The sum over all elements for the first term generates jumps in the fluxes and entropy variables where we also use the uniqueness of the numerical surface flux function yielding
\begin{equation}\label{eq:surfaceTermsOnK2}
\sum\limits_{k=1}^K \int\limits_{\partial E_k,N} \statevec{W}^T \left[\bigcontravec{F}^{\ec}-\bigcontravec{F}^a \right]\cdot \spacevec{n} \dS = \sum\limits_{\interiorfaces}\int\limits_{N} \left\{\left(\bigcontravec{F}^{\ec}\right)^T\jump{\statevec{W}} - \jump{\left(\bigcontravec{F}^a \cdot\spacevec{n}\right)^T\statevec{W}}\dS\right\}.
\end{equation}
First, we examine the behavior of the GLM part of the entropy conservative flux at the interfaces that come from \eqref{eq:surfaceTermsOnK}.
\begin{lem}[Entropy contribution of GLM surface terms]\label{lem:GLM_EC_surf}~\newline
The contribution from the GLM part of the entropy conservative scheme vanishes at element interfaces, i.e.,
\begin{equation}
\int\limits_{N} \left\{\left(\bigcontravec{F}^{\ec,\mathrm{GLM}}\right)^T\jump{\statevec W} - \jump{\left(\bigcontravec{F}^{a,\mathrm{GLM}}\cdot\spacevec{n}\right)^T\,\statevec{W}}\right\}\dS = 0.
\end{equation}
\end{lem}
\begin{proof}
This follows directly from the definition of the GLM components of the entropy conservative flux \eqref{eq:GLMPartDiscrete} 
\begin{equation}
\left(\bigcontravec{F}^{\ec,\mathrm{GLM}}\right)^T\jump{\statevec W} - \jump{\left(\bigcontravec{F}^{a,\mathrm{GLM}}\cdot\spacevec{n}\right)^T\,\statevec{W}} = \jump{\statevec W}^T\contrastatevec{F}^{\ec,\mathrm{GLM}} - \jump{\spacevec{\Psi}^{\mathrm{GLM}}\cdot\spacevec{n}} = 0,
\end{equation}
where we note that because the arguments are real-valued, we can switch the ordering in the inner product.
\end{proof}

We leave the remaining contributions of the Euler and ideal MHD components to later in this section, since we first define $\Bstar$ and examine the contribution of the second term from \eqref{eq:surfaceTermsOnK}. What we will find is that the surface contribution of the non-conservative terms generates an additional boundary term that cancels an extraneous term left over from the analysis of the ideal MHD part of the advective fluxes.
\begin{lem}[Discretization of the non-conservative surface term]\label{lem:nonConsSurf}~\newline
For the second term in \eqref{eq:surfaceTermsOnK} we define 
\begin{equation}\label{eq:surfNonCons}
\Bstar\cdot\spacevec{n} = \JanD^{\ma}\avg{\contraspacevec{B}\cdot\spacevec{n}},
\end{equation}
to obtain the total non-conservative surface contribution
\begin{equation}
 \sum\limits_{k=1}^K\int\limits_{\partial E_k,N}  \statevec{W}^T \left[\Bstar - \JanD\contraspacevec{B}\right]\cdot \spacevec{n} \dS = - \sum\limits_{\interiorfaces}\int\limits_{N} \avg{\wDotJan}\jump{\contraspacevec{B}\cdot\spacevec n}\dS.
\end{equation}
\end{lem} 
\begin{proof}
We first substitute the definition \eqref{eq:surfNonCons} into the second term of \eqref{eq:surfaceTermsOnK}, where, for clarity, we explicitly state which values come from the current element $k$ and what comes from the neighbors
\begin{equation}\label{eq:firstStepNonCons}
\int\limits_{\partial E_k,N}  \statevec{W}^T \left[\Bstar - \JanD\contraspacevec{B}\right]\cdot \spacevec{n} \dS = \int\limits_{\partial E_k,N}  \left(\statevec{W}^{\ma}\right)^T \left[\JanD^{\ma}\avg{\contraspacevec{B}\cdot\spacevec{n}} - \JanD^{\ma}\left(\contraspacevec{B}^{\ma}\cdot\spacevec{n}\right)\right] \dS.
\end{equation}
Note that the values of $\statevec{W}$ and $\JanD$ in the contribution \eqref{eq:firstStepNonCons} are evaluated from the current element, so we have a discrete version of the property \eqref{eq:contractSource}
\begin{equation}
\left(\statevec{W}^{\ma}\right)^T\JanD^{\ma} = \wDotJan^{\ma}.
\end{equation}
Thus,
\begin{equation}\label{eq:secondStepNonCons}
\int\limits_{\partial E_k,N}  \left(\statevec{W}^{\ma}\right)^T \left[\JanD^{\ma}\avg{\contraspacevec{B}\cdot\spacevec{n}} - \JanD^{\ma}\left(\contraspacevec{B}^{\ma}\cdot\spacevec{n}\right)\right] \dS
= 
\int\limits_{\partial E_k,N}  \wDotJan^{\ma} \left[\avg{\contraspacevec{B}\cdot\spacevec{n}} - \left(\contraspacevec{B}^{\ma}\cdot\spacevec{n}\right)\right] \dS.
\end{equation}
Next, we expand the arithmetic mean at the four interfaces of element $k$ in \eqref{eq:secondStepNonCons} to get
\begin{equation}\label{eq:thirdStepNonCons}
\int\limits_{\partial E_k,N}  \wDotJan^{\ma} \left[\avg{\contraspacevec{B}\cdot\spacevec{n}} - \left(\contraspacevec{B}^{\ma}\cdot\spacevec{n}\right)\right] \dS
=
\frac{1}{2}\int\limits_{\partial E_k,N} \wDotJan^{\ma} \left[\contraspacevec{B}^{\sl} - \contraspacevec{B}^{\ma}\right]\cdot\spacevec{n} \dS.
\end{equation}

The total surface contribution of \eqref{eq:thirdStepNonCons} requires delicate consideration due to the inherent non-uniqueness of the non-conservative term at the interface. Each interface actually contributes twice to the contracted DG approximation and it is important to take into account that the normal vector is outward pointing. In the sum over all elements we inspect the left and right contributions of one particular interface from \eqref{eq:thirdStepNonCons}
\begin{equation}\label{eq:singleFaceTerm}
\resizebox{0.925\hsize}{!}{$
\begin{aligned}
-\frac{1}{2}\left\{\wDotJan^{\ma}\left(\contraspacevec{B}^{\sl}-\contraspacevec{B}^{\ma}\right)\cdot\spacevec{n}\right\}
+\frac{1}{2}\left\{\wDotJan^{\sl}\left(\contraspacevec{B}^{\ma}-\contraspacevec{B}^{\sl}\right)\cdot\spacevec{n}\right\}
&=-\frac{1}{2}\jump{\wDotJan\left(\contraspacevec{B}\cdot\spacevec{n}\right)} +\frac{1}{2}\left\{\wDotJan^{\sl}\left(\contraspacevec{B}^{\ma}\cdot\spacevec{n}\right)-\wDotJan^{\ma}\left(\contraspacevec{B}^{\sl}\cdot\spacevec{n}\right)\right\}\\[0.1cm]
&=-\avg{\wDotJan}\jump{\contraspacevec{B}\cdot\spacevec{n}},
\end{aligned}$}
\end{equation}
where all other interfaces are handled in an analogous fashion. 

We then determine the desired result
\begin{equation}
 \sum\limits_{k=1}^K\int\limits_{\partial E_k,N}  \statevec{W}^T \left[\Bstar - \JanD\contraspacevec{B}\right]\cdot \spacevec{n} \dS = - \sum\limits_{\interiorfaces}\int\limits_{N} \avg{\wDotJan}\jump{\contraspacevec{B}\cdot\spacevec n}\dS.
\end{equation}

\end{proof}

\begin{rem}
The definition of $\Bstar$ in \eqref{eq:surfNonCons} can be interpreted as a partial split form of the non-conservative term. It is easily seen to be consistent, but it is not symmetric with respect to its arguments because it is only ``half'' of the full splitting. This was noted in the above proof and revealed that each interface contributes twice in the contracted DG approximation, one from the left and another from the right.
\end{rem}

\begin{rem}
The definition \eqref{eq:surfNonCons} is a re-contextualization for the non-conservative surface term discretization of previous entropy stable DG methods for the ideal MHD equations \cite{Liu2017,Valencia2017} into the split form DG context.
\end{rem}

The sum over all elements on the third term in \eqref{eq:surfaceTermsOnK} generates a jump in the entropy fluxes
\begin{equation}\label{eq:surfEntFluxes}
 \int\limits_{\partial E_k,N} \left(\contraspacevec{F}^\ent\cdot\spacevec n\right) \dS = \int\limits_{N}\jump{\left(\contraspacevec{F}^{\ent}\cdot\spacevec{n}\right)}\dS.
\end{equation}
Now, with the results of Lemmas \ref{lem:GLM_EC_surf} and \ref{lem:nonConsSurf} as well as the results \eqref{eq:surfaceTermsOnK2} and \eqref{eq:surfEntFluxes} we can address the remaining contributions of the Euler and ideal MHD components at the surface:
\begin{cor}[Entropy contributions of total advective surface terms]\label{cor:EC_surf}~\newline
Summing over all elements in \eqref{eq:surfaceTermsOnK} shows that the contribution of the advective and non-conservative terms on the surface cancel, meaning
\begin{equation}
\sum_{k=1}^K \Gamma_k = 0.
\end{equation}
\end{cor}
\begin{proof}
We note that from Lemma \ref{lem:GLM_EC_surf} we have accounted for the cancellation of the GLM terms. Similar to the volume term analysis in Corollary \ref{cor:EC_vol} we again separate the contributions of the Euler and ideal MHD terms. It is immediate that the Euler terms drop out from the definition of the entropy flux potential for the Euler part \eqref{eq:EulerEntFluxPot} and the constriction of the entropy conserving flux \eqref{eq:EulerPartDiscrete}
\begin{equation}
\left(\bigcontravec{F}^{\ec,\text{Euler}}\right)^T\jump{\statevec W} - \jump{\left(\bigcontravec{F}^{a,\text{Euler}}\cdot\spacevec{n}\right)^T\,\statevec{W}} + \jump{\contraspacevec{F}^\ent\cdot\spacevec{n}} = 
\jump{\statevec W}^T\bigcontravec{F}^{\ec,\text{Euler}} - \jump{\spacevec{\Psi}^{\text{Euler}}\cdot\spacevec{n}} = 0.
\end{equation}
The ideal MHD contributions require more manipulation. First, we make use of the entropy flux potential \eqref{eq:MHDEntFluxPot} to write
\begin{equation}\label{eq:manipulateMHD}
\resizebox{0.925\hsize}{!}{$
\begin{aligned}
\left(\bigcontravec{F}^{\ec,\text{MHD}}\right)^T\jump{\statevec W} - \jump{\left(\bigcontravec{F}^{a,\text{MHD}}\cdot\spacevec{n}\right)^T\,\statevec{W}} &=
\jump{\statevec W}^T\bigcontravec{F}^{\ec,\text{MHD}} - \jump{\spacevec{\Psi}^{\text{MHD}}\cdot\spacevec{n}} + \jump{\wDotJan\left(\contraspacevec{B}\cdot\spacevec{n}\right)}\\
&= \jump{\statevec W}^T\bigcontravec{F}^{\ec,\text{MHD}} - \jump{\spacevec{\Psi}^{\text{MHD}}\cdot\spacevec{n}} + \avg{\contraspacevec{B}\cdot\spacevec{n}}\jump{\wDotJan} + \avg{\wDotJan}\jump{\contraspacevec{B}\cdot\spacevec{n}},
\end{aligned}$}
\end{equation}
where we use a property of the jump operator
\begin{equation}\label{eq:jumpProperty}
\jump{ab} = \avg{a}\jump{b} + \avg{b}\jump{a}.
\end{equation}
We see that the first three terms on the last line of \eqref{eq:manipulateMHD} are the entropy conservative flux condition \eqref{eq:MHDPartDiscrete} and cancel. This leaves the remainder term
\begin{equation}\label{eq:remainderTermMHD}
\left(\bigcontravec{F}^{\ec,\text{MHD}}\right)^T\jump{\statevec W} - \jump{\left(\bigcontravec{F}^{a,\text{MHD}}\cdot\spacevec{n}\right)^T\,\statevec{W}} 
= \avg{\wDotJan}\jump{\contraspacevec{B}\cdot\spacevec{n}}.
\end{equation}
This term is identical to the surface contribution of the non-conservative term from Lemma \ref{lem:nonConsSurf} but with opposite sign. These final two terms cancel and yield the desired result
\begin{equation}
\sum_{k=1}^K \Gamma_k = 0.
\end{equation}
\end{proof}

Summarizing this section, we have demonstrated that, neglecting the GLM damping source term (Remark \ref{ES_damp}) and assuming periodic boundary conditions, all of the advective and non-conservative contributions in the approximation \eqref{eq:schemeFinalcontr} cancel in entropy space. 

\subsection{Analysis of the resistive parts}\label{Sec:ResParts}

Lastly, since the discussion of the advective and non-conservative parts is now complete, we focus on the resistive parts, namely the last row of the first equation in \eqref{eq:schemeFinalcontr}. Again, we first have to select appropriate numerical fluxes at the interfaces. Thus, we use the computationally simple Bassi-Rebay (BR1) type approximation \cite{Bassi&Rebay:1997:B&F97} in terms of the discrete entropy variables and gradients \cite{Gassner2017}
\begin{equation}\label{eq:BR1Fluxes}
\bigcontravec F^{v,*}\cdot\spacevec n = \avg{\bigcontravec{F}^v\cdot\spacevec n},\quad \statevec{W}^*=\avg{\statevec{W}}.
\end{equation}
With the knowledge of the previous section we are equipped to prove the second main result of this work.
\begin{thm}[Discrete entropy stability of the DGSEM for the resistive GLM-MHD equations]\label{thm_discrete}~\newline
The DGSEM for the resistive GLM-MHD equations \eqref{eq:schemeFinal} with 
\begin{equation}
\Bstar\cdot\spacevec{n} = \JanD^{\ma}\avg{\contraspacevec{B}\cdot\spacevec{n}}, \quad\bigcontravec{F}^{a,\ast} = \bigcontravec{F}^{a,\#} = \bigcontravec{F}^\ec,
\end{equation}
and the viscous interface fluxes \eqref{eq:BR1Fluxes} is entropy stable for periodic boundary conditions.
\end{thm}
\begin{proof}
From Corollaries \ref{cor:EC_vol}, \ref{cor:EC_surf} we know that the volume, surface and non-conservative terms of the advective portions of the resistive GLM-MHD equations cancel in entropy space. The remaining parts of the contracted DG approximation are
\begin{equation}\label{eq:schemeFinalInProof}
 \begin{aligned}
\iprodN{J\statevec U_t,\statevec{W} } 
&= \iprodN{\spacevec{\mathbb{D}}^S\bigcontravec F^v,\statevec{W} }+\int\limits_{\partial E,N}\statevec{W}^T\left\{\left(\bigcontravec F^{v,*} - \bigcontravec F^{v}\right)\cdot\spacevec n \right\}\dS + \left\langle J \statevec{R},\statevec{W} \right\rangle_N\hfill, \\
\iprodN{J\bigstatevec Q,\bigcontravec F^v  }
&= \int\limits_{\partial E,N} {{{\statevec W}^{*,T}}\left( {\bigcontravec F^v }  \cdot \spacevec n\right)\dS}  
- \iprodN{\statevec W, {\spacevec{\mathbb{D}}^S \bigcontravec F^v } }. \hfill \\ 
 \end{aligned}
\end{equation}
We consider the first term of the second equation and insert the alternate form of the viscous flux rewritten in terms of the gradient of the entropy variables as in the continuous analysis \eqref{Kgradient}. We use the known property that the viscous flux matrix $\twentysevenMatrix{K}$ is symmetric positive semi-definite for the resistive MHD equations to see that
\begin{equation}
\label{eq:disc_visc_volint_estimate}
\iprodN{J\bigstatevec{Q},\bigstatevec{F}_v} = J \, \iprodN{\bigstatevec{Q},\twentysevenMatrix{K}\,\bigstatevec{Q}}\geq  0.
\end{equation}
Next, we insert the second equation of \eqref{eq:schemeFinalInProof} into the first and use the estimate \eqref{eq:disc_visc_volint_estimate} to get
\begin{equation}\label{eq:schemeFinalInProof2}
\iprodN{J\statevec U_t,\statevec{W} } 
\leq \int\limits_{\partial E,N}\left\{\statevec{W}^T\left(\left(\bigcontravec F^{v,*} - \bigcontravec F^{v}\right) \cdot\spacevec n \right) + {\statevec W^{*,T}}\left( {\bigcontravec F^v }  \cdot \spacevec n\right)\right\}\dS + \left\langle J \statevec{R},\statevec{W} \right\rangle_N
\end{equation}
From Remark \ref{ES_damp} we know, that we can also ignore the discrete damping source term without violating the inequality. After summing over all elements we
can replace the left hand side by the total entropy derivative according to \eqref{totalEntr} and obtain interface jumps on the right hand side as in the previous analysis
\begin{equation}\label{eq:schemeFinalInProof3}
\frac{d \overline{S}}{d t}
\leq \sum\limits_\interiorfaces \int\limits_{N}\left\{\left(\bigcontravec F^{v,*}\cdot\spacevec n\right)^T\jump{\statevec{W}} - \jump{\left(\bigcontravec F^{v}\cdot\spacevec n\right)^T\statevec{W}}  + \left(\statevec{W}^{*}\right)^T\jump{\bigcontravec F^v \cdot \spacevec n}\right\}\dS.
\end{equation}

With the choice of the numerical surface fluxes for the viscous terms \eqref{eq:BR1Fluxes} we see that each term on the right hand side of \eqref{eq:schemeFinalInProof3} takes the form 
\begin{equation}\label{eq:viscFluxInt}
\resizebox{0.925\hsize}{!}{$
\left(\bigcontravec F^{v,*}\cdot\spacevec n\right)^T\jump{\statevec{W}} - \jump{\left(\bigcontravec F^{v}\cdot\spacevec n\right)^T\statevec{W}}  + \left(\statevec{W}^{*}\right)^T\jump{\bigcontravec F^v \cdot \spacevec n} = \avg{\bigcontravec F^{v}\cdot\spacevec n}^T\jump{\statevec{W}} - \jump{\left(\bigcontravec F^{v}\cdot\spacevec n\right)^T\statevec{W}}  + \avg{\statevec{W}}^T\jump{\bigcontravec F^v \cdot \spacevec n}$}
\end{equation}
and repeated use of the jump property \eqref{eq:jumpProperty} gives us
\begin{equation}
 \jump{(\bigcontravec{F}^{v}\cdot\spacevec n)^T\statevec{W}} = \avg{\bigcontravec{F}^{v}\cdot\spacevec n}^T\jump{\statevec{W}} + \jump{\bigcontravec{F}^{v}\cdot\spacevec n}^T\avg{\statevec{W}},
\end{equation}
so that the contribution of the viscous numerical fluxes at the interior faces \eqref{eq:viscFluxInt} vanishes exactly. Again, assuming periodic boundary conditions for the viscous fluxes, too, the final discrete entropy statement is
\begin{equation}\label{eq:entr_stab_bounds}
\frac{d \overline{S}}{d t} \leq 0,
\end{equation}
which proves discrete entropy stability for the resistive GLM-MHD equations.
\end{proof}

\begin{rem}
We can introduce additional dissipation without violating the entropy condition by replacing the EC fluxes at element interfaces by the ES fluxes, e.g., \eqref{ES_flux}.
\end{rem}

\begin{rem}
We presented the proofs of entropy stability for the DG approximation on two-dimensional uniform Cartesian meshes. However, for this simplified mesh assumption, the extension to three spatial dimensions is straightforward. This is because the derivatives in the DG approximation are decoupled in each spatial direction. Therefore, the proofs presented throughout this section hold in three dimensions by adding a $\zeta-$direction in the volume contributions and surface coupling is done at neighboring two-dimensional ``faces'' rather than at one dimensional ``edges''.
\end{rem}

\section{Numerical results}\label{Sec:Num}

In this section we present numerical tests to demonstrate the high-order accuracy, entropy conservation/stability, divergence cleaning capability and increased robustness of the numerical methods derived in this work. We begin in Sec. \ref{Sec:Conv_MS} with the method of manufactured solutions to verify the high-order convergence for the resistive GLM-MHD scheme. Section \ref{Sec:EC}, though an academic test case, numerically shows the entropy conservation property for the DG approximation of the ideal GLM-MHD using a weak shock tube with periodic boundaries. Next, the divergence cleaning properties of the high-order ideal GLM-MHD scheme are given in Sec. \ref{Sec:Gauss}. Finally, we provide examples, in which every piece of the presented numerical solver are exercised, to demonstrate the increased robustness of the entropy stable DG approximation for the resistive GLM-MHD equations. In these investigations we will show the necessity of the entropy stable framework in conjunction with GLM hyperbolic divergence cleaning to provide numerical stability for problems with very weak diffusivity. Specifically, in two spatial dimensions, we use a viscous version of the well-known Orszag-Tang vortex, e.g. \cite{altmann2012}, and in three spatial dimensions increased robustness is presented for an under-resolved turbulence computation using an extension of the Taylor-Green vortex to MHD models, e.g. \cite{Lee2010,Brachet2013}.

Unless otherwise stated, we set \texttt{CFL}=\texttt{DFL}=0.5, the damping parameter $\alpha=0$ and the GLM propagation speed $c_h$ to be proportional to the maximum advective wave speed. Also, a suitable choice for the initial value of $\psi$ is $\psi(\spacevec{x},t=0) = 0$. 

\subsection{Manufactured solution test for viscous equations}\label{Sec:Conv_MS}

In order to verify the high-order approximation of the entropy stable DG discretization \eqref{eq:schemeFinal} for the resistive GLM-MHD system \eqref{resGLMMHD}, we run a convergence test with the method of manufactured solutions. To do so, we assume a solution of the form
\begin{equation}
\statevec{u} = \left[h,h,h,0,2h^2,h,-h,0,0\right]^T \text{ with } h = h(x,y,t) = \sin(2\pi(x+y)-4t)+4.
\label{manufac}
\end{equation}
This generates a residual for the resistive GLM-MHD system defined as
\begin{equation}
\statevec{u}_t + \vec{\nabla} \cdot \bigstatevec{f}^a(\statevec{u}) - \vec{\nabla} \cdot \bigstatevec{f}^v(\statevec{u},\vec{\nabla} \statevec{u}) = \begin{pmatrix} h_t + 2h_x \\ h_t + h_x + 4hh_x \\ h_t + h_x + 4hh_x \\ 0 \\ 4hh_t + 16hh_x - 2h_x - 4\resistivity(h_x^2+h h_{xx})-4\viscosity h_{xx}/\text{Pr} \\ h_t + 2h_x - 2\resistivity h_{xx} \\ -h_t - 2h_x + 2\resistivity h_{xx} \\ 0 \\ 0 \end{pmatrix}
\label{residual}
\end{equation}
for $\gamma = 2$ and $\text{Pr} = 0.72$. To solve the inhomogeneous problem we subtract the discrete residual from the approximate solution after each Runge-Kutta step. We run the test case on the periodic domain $\Omega = \left[0,1\right]^2$ up to the final time $T=0.5$. Furthermore, we set $\viscosity=\resistivity=0.05$ and finally obtain the convergence results illustrated in Tables \ref{EOCres3} and \ref{EOCres4}. The errors of $v_2$ and $B_2$ are identical to the ones of $v_1$ and $B_1$, whereas $v_3$ and $B_3$ are equal to zero for all time.

\begin{table}[h!]
	\centering
		\begin{tabular}{c|c|c|c|c|c}
			 \small $\Delta x$ & $L^2(\rho)$ & $L^2(v_1)$ & $L^2(p)$ & $L^2(B_1)$ & $L^2(\psi)$ \\
			\hline 
			$1/5$ & 3.45E-03 &	1.17E-03 &	1.38E-02	& 1.76E-03 &	1.76E-03\\
			\hline
			$1/10$ & 2.28E-04	& 8.59E-05	&	7.95E-04 &	1.02E-04	& 1.31E-04\\
			\hline
			$1/20$ & 1.48E-05 &	5.82E-06	&	5.23E-05 &	5.40E-06	&	8.64E-06\\
			\hline
			$1/40$ & 8.25E-07	& 3.33E-07	&	3.12E-06	& 3.07E-07	&	5.48E-07\\
			\hline
			\hline 	
			avg EOC & 4.01	& 3.93	&	4.04	& 4.16	&	3.88\\
			\hline 
		\end{tabular}
		\caption{$L^2$-errors and EOC of manufactured solution test for resistive GLM-MHD and $N=3$.}\label{EOCres3}
\end{table}
\begin{table}[h!]
	\centering
		\begin{tabular}{c|c|c|c|c|c}
			 \small $\Delta x$ & $L^2(\rho)$ & $L^2(v_1)$ & $L^2(p)$ & $L^2(B_1)$ & $L^2(\psi)$ \\
			\hline 
			$1/5$ & 1.87E-04	& 8.56E-05	&	1.03E-03	& 9.76E-05	&	1.14E-04\\
			\hline
			$1/10$ & 5.38E-06	& 2.74E-06	&	4.01E-05	& 2.75E-06 &	3.39E-06\\
			\hline
			$1/20$ & 1.98E-07 &	7.34E-08 &	1.46E-06	& 8.02E-08	&	1.04E-07\\
			\hline
			$1/40$ & 7.66E-09 &	1.57E-09	&	5.80E-08	& 2.48E-09	& 3.23E-09\\
			\hline
			\hline 	
			avg EOC & 4.86 &	5.25	&	4.70	& 5.09	&	5.04\\
			\hline 
		\end{tabular}
		\caption{$L^2$-errors and EOC of manufactured solution test for resistive GLM-MHD and $N=4$.}\label{EOCres4}
\end{table}

We note, if we apply an entropy conservative approximation to this test case, the convergence order would exhibit an odd/even effect depending on the polynomial degree. This phenomenon of optimal convergence order for even and sub-optimal convergence order for odd polynomial degrees has been previously observed for entropy conservative DG approximations, e.g. \cite{gassner_skew_burgers,Gassner:2016ye}.

\subsection{Entropy conservation}\label{Sec:EC}

In this section we focus on the entropy conservation properties of the method. As such, we deactivate the numerical dissipation introduced by the interface stabilization terms \eqref{ES_flux} and set $\viscosity=\resistivity=0$ to remove the resistive terms, because entropy conservation only applies to the ideal GLM-MHD equations. 

In contrast to the previous convergence investigation, it does not make sense to measure entropy conservation with well-resolved smooth solutions, as in this case the errors in the DGSEM would converge spectrally fast to entropy conservation. Thus, to make this conservation test challenging we consider a oblique shock tube test with periodic boundary conditions. As the flow evolves physical entropy dissipation occurs at the shock. However, the high-order DG scheme artificially preserves entropy by construction and, near discontinuities, energy is re-distributed at the smallest resolvable scale \cite{Mishra2011}. This is why we select this test to verify the theoretical properties of the derived DG scheme.

We define the test case in the periodic box $\Omega = \left[0,1\right]^2$ initialized by a small diagonal shock (see e.g. \cite{Toth2000}):
\begin{table}[h!]
	\centering
		\begin{tabular}{c|c|c|c|c|c|c|c|c}
			  & $\rho$ & $v_1$ & $v_2$ & $v_3$ & $p$ & $B_1$ & $B_2$  & $B_3$ \\
			\hline
			$x<y$ & 1.0 &	0.0 &	0.0 &	0.0 &	1.0 &	$2.0/\sqrt{4.0\pi}$ &	$4.0/\sqrt{4.0\pi}$ & $2.0/\sqrt{4.0\pi}$\\
			\hline
			$x\geq y$ & 1.08 &	0.6 & 0.01 &	0.5 & 0.95 &	$2.0/\sqrt{4.0\pi}$ &	$3.6/\sqrt{4.0\pi}$ &	$2.0/\sqrt{4.0\pi}$\\			
			\hline 
		\end{tabular}
		\caption{Initialized primitives for the entropy conservation test.}
\end{table}\\
Furthermore, we set $\gamma\!=\! 5/3,\; \Delta x\! =\! \Delta y \!=\! 1/20$ and  $T\!=\! 0.5$.

The entropy conservative DGSEM for the ideal GLM-MHD equations is essentially dissipation-free. In the case of discontinuous solutions, the only dissipation introduced into the approximation is through the time integration scheme. Hence, we can use the error in the total entropy as a measure of the temporal convergence of the method. We know that by shrinking the \texttt{CFL} number, the dissipation of the time integration scheme is lessened and entropy conservation is captured more accurately, e.g. \cite{Fjordholm2011}.

In the log-log plot, Figure \ref{RZ}, we illustrate the numerical results for $N=2$ and $N=3$, which confirm this precise behavior. There is a reduction of the entropy error according to the fourth order time integration scheme for decreasing \texttt{CFL} numbers. That is, halving the \texttt{CFL} number leads to a factor of 16 reduction in the entropy error. In fact, if the \texttt{CFL} number is taken small enough, the entropy conservation error is driven to the order of machine precision. This demonstrates the discrete entropy conservation property of the scheme without viscous terms modulo the discretization errors introduced by the time integration scheme.

\begin{figure}[h!]
\begin{center}
\includegraphics[width=9cm]{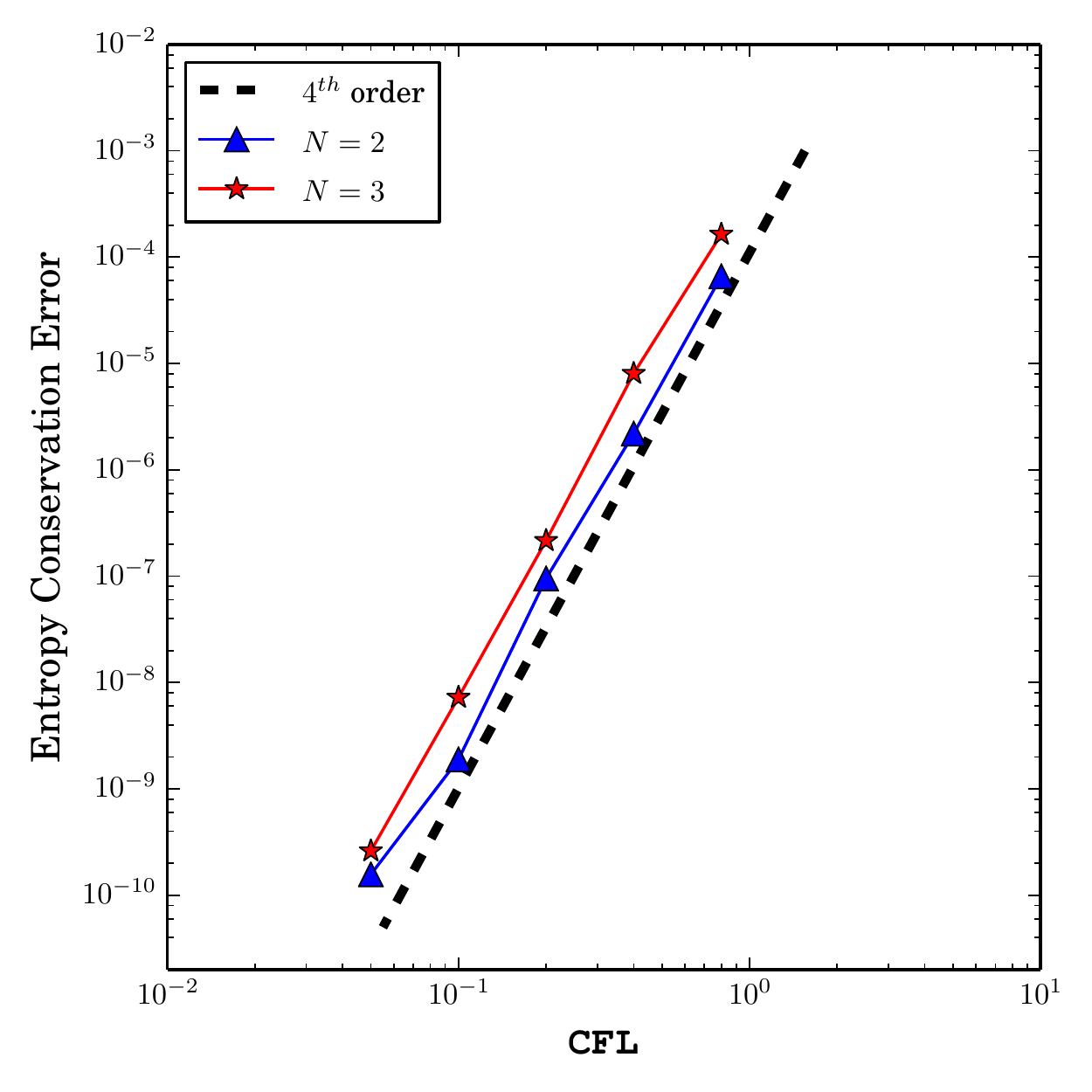}
\caption{Log-log plot of entropy conservation error for $N=2$ and $N=3$ on a uniform $20\times 20$ Cartesian mesh and final time T=0.5.}
\label{RZ}
\end{center}
\end{figure}

\subsection{Divergence Cleaning}\label{Sec:Gauss}

In order to demonstrate the reduction of the divergence error in the magnetic field, we use a maliciously chosen non-divergence-free initialization in $\Omega = \left[-1,1\right]^2$ defined by a Gaussian pulse in the $x-$component of the magnetic field proposed in \cite{altmann2012}:
\begin{equation*}
 \rho(x,y,0) = 1, ~~~ E(x,y,0) = 6, ~~~ B_1(x,y,0) = \exp\left(-\frac{1}{2}\left[\frac{x^2+y^2}{0.11^2}\right]\right).
\end{equation*}
The other initial values are set to zero and the boundaries are periodic. Again we set $\gamma = 5/3$ and turn off physical viscosity. In Figure \ref{DivError} we illustrate the time evolution of the normalized discrete divergence error measured in terms of $\|\vec{\nabla} \cdot \spacevec{B}\|_{L^2(\Omega)}$ for $N=3$ and $20\times 20$ elements. Here, we show the simulation results of the ideal GLM-MHD approximation, in which the divergence error is solely propagated through the domain. However, due to the periodic nature of the boundaries, it is known that the divergence errors will simply advect back into the domain \cite{Dedner2002} with only minimal damping due to the high-order nature of the DG scheme. As such, we provide a study comparing the no divergence cleaning case against the GLM with additional damping and varying the value of $\alpha$ in \eqref{damping}. For this periodic test problem we see that extra damping is necessary to control errors in Figure \ref{DivError}. These high-order results reinforce a similar study done in the finite volume context \cite{Derigs2017}.

\begin{figure}[h!]
\begin{center}
\includegraphics[width=9cm]{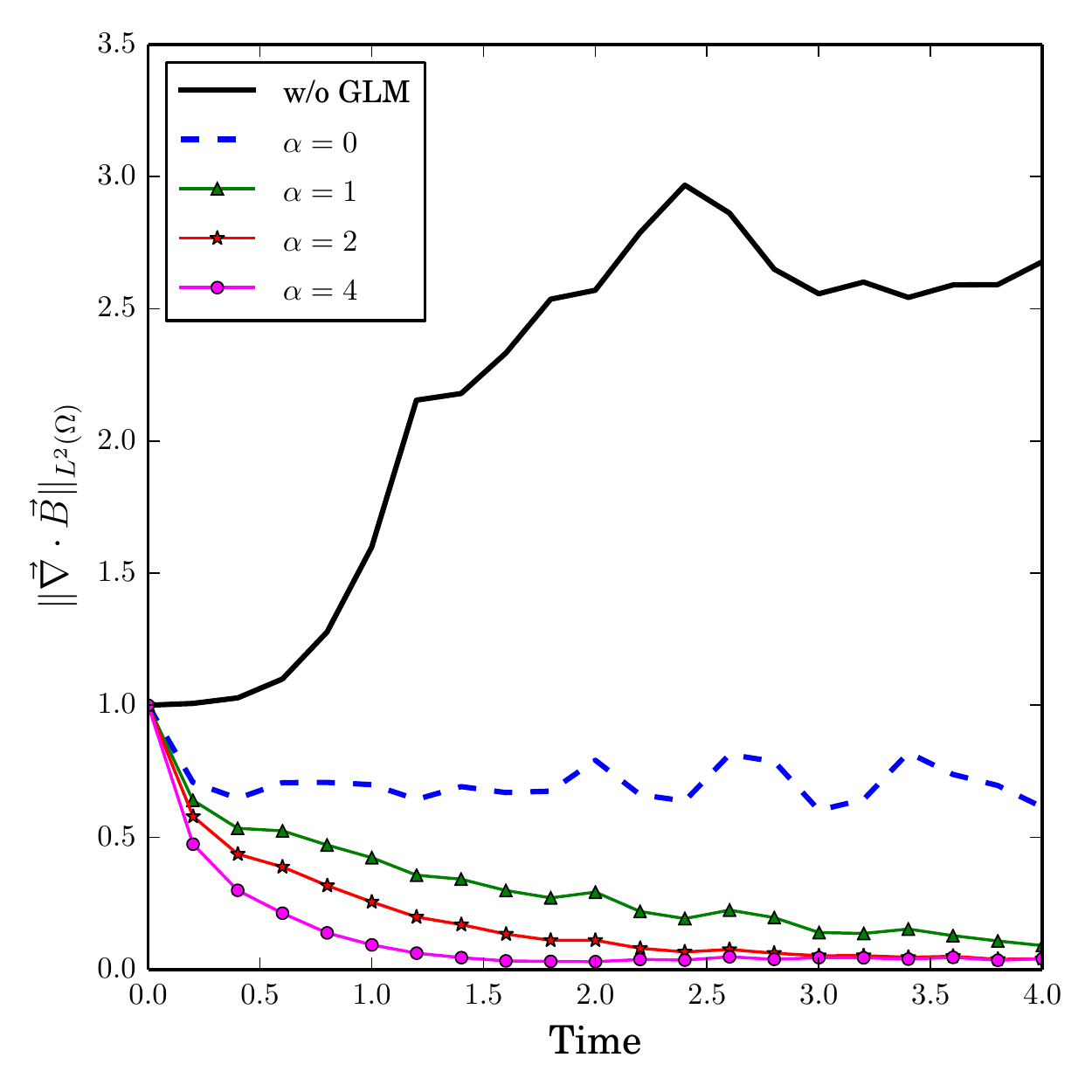}
\caption{Temporal evolution of the normalized discrete $L^2$ error in the divergence-free condition for $N=3$ in each spatial direction on $20\times 20$ uniform elements.}
\label{DivError}
\end{center}
\end{figure}

\subsection{Viscous Orszag-Tang vortex}\label{Sec:OTV}

Next, we use a viscous version of the Orszag-Tang vortex \cite{altmann2012,Orszag1979,Derigs2017} to demonstrate the increased robustness of the entropy stable approximation including GLM divergence cleaning. The initial conditions are simple and smooth, but evolve to contain complex structures and energy exchanges between the velocity and magnetic fields. The domain is $\Omega = [0,1]^2$ with periodic boundary conditions and initial data
 \begin{align}
\rho = & ~ 1  ~~~~~~~~~~~~\, v_1 = -\sin(2\pi y) ~~~~~~~~~~~~\, v_2 = \sin(2\pi x) \\[0.1cm]
p = & ~ \frac{1}{\gamma} ~~~~~~~~~~~ B_1 = -\frac{1}{\gamma} \sin(2\pi y) ~~~~~~~~~\, B_2 = \frac{1}{\gamma} \sin(4\pi x)
\end{align}
with $\gamma = 5/3$. To include diffusivity in the simulation we select the Prandtl number to be $\mathrm{Pr}=0.72$ and the viscosity and resistivity parameters to be
\begin{equation}\label{eq:resParameters}
\viscosity = 8.5\times 10^{-4},\quad\resistivity = 10^{-5}.
\end{equation}
This selection of the diffusive coefficients corresponds to a kinematic Reynolds number (Re) of approximately $1170$ and a magnetic Reynolds number ($\mathrm{Re_m}$) of $100,000$. The initial conditions evolve to a final time of $T=0.5$.

Moreover, the simulation uses a polynomial order $N=7$ in each spatial direction on a $20\times 20$ Cartesian grid. We run three variants of the DG approximation to demonstrate the necessity of the entropy stable scheme as well as the GLM modification for numerical stability of this particular test case configuration in Table \ref{tab:crashTestsOT}.
\begin{table}[h!]
	\centering
		\begin{tabular}{| l | c |}
		         \hline
			 \textsc{Configuration} & $N=7$, $20 \times 20$ mesh \\
			\hline
			\hline 
			Standard DGSEM with GLM divergence cleaning & \textbf{\textit{crash}} \\
			\hline
			Entropy stable DGSEM without GLM divergence cleaning & \textbf{\textit{crash}} \\
			\hline
			Entropy stable DGSEM with GLM divergence cleaning & \textbf{\textit{result}}  \\
			\hline 
		\end{tabular}
		\caption{Comparison of the numerical stability of the standard DGSEM against the entropy stable version with and without GLM divergence cleaning, applied to the viscous Orszag-Tang vortex problem.}\label{tab:crashTestsOT}
\end{table}

These results demonstrate that the entropy stable formulation as well as numerical treatment of the divergence-free constraint are needed to create a robust scheme for this configuration. 
We shrink the time step by setting \texttt{CFL}=\texttt{DFL}=0.25 and find the same numerical stability results for the three configurations presented in Table \ref{tab:crashTestsOT}. This reinforces that the numerical instabilities in the approximate flow are caused by errors other than those introduced by the time integration scheme. For the entropy stable DG simulation with GLM divergence cleaning we illustrate the density together with contour lines of the magnetic field at the final time in Figure \ref{OTV}. We also provide a time-dependent plot of the normalized total entropy for the case \texttt{CFL}=\texttt{DFL}=0.5 to show the entropy stability of the entire approximation in Figure \ref{EE}.

We note that the high-order entropy stable approximation is not guaranteed to be oscillation-free near shocks, e.g. \cite{wintermeyer2017}. Thus, there are still observable, albeit small, numerical artifacts in the approximate solution (Figure \ref{OTV}). But, due to the viscous and resistive terms, the entropy stable DGSEM applied to this setup of the Orszag-Tang problem can run without any additional shock capturing method.

\begin{figure}[h!]
\begin{center}
\includegraphics[width=13cm]{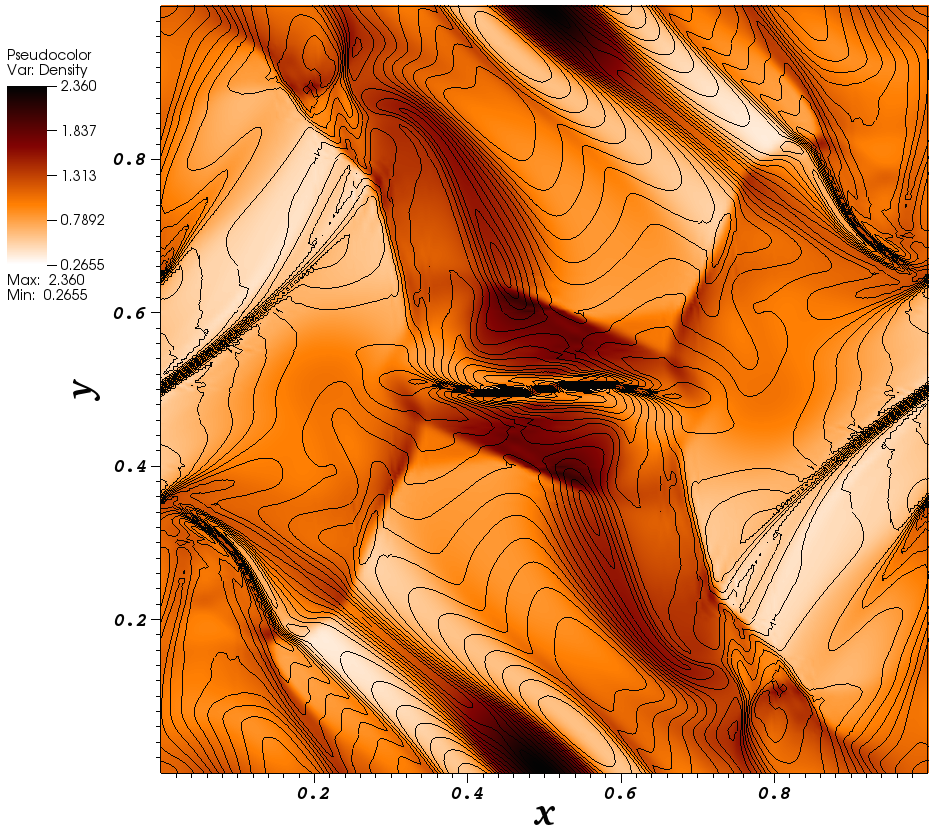}
\caption{Density pseudocolor plot with overlaid magnetic field lines for the viscous Orszag-Tang vortex at $T=0.5$ with $N=7$ in each spatial direction on $20\times 20$ element mesh and diffusivity coefficients \eqref{eq:resParameters}.}
\label{OTV}
\end{center}
\end{figure}

\begin{figure}[h!]
\begin{center}
\includegraphics[width=8cm]{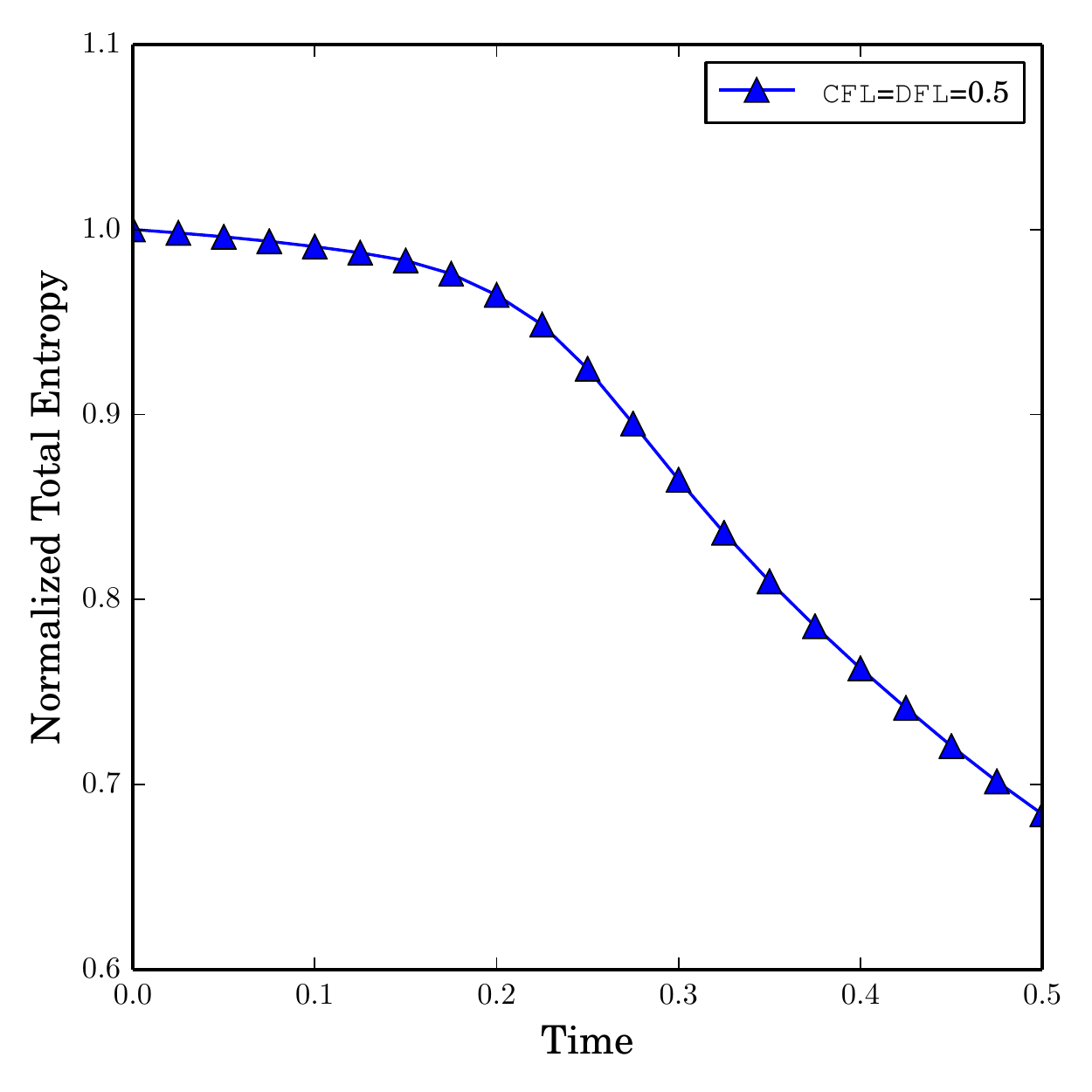}
\caption{Time evolution plot of the total entropy in the viscous Orszag-Tang vortex with $N=7$ on $20\times 20$ elements with diffusivity coefficients \eqref{eq:resParameters}.}
\label{EE}
\end{center}
\end{figure}

\subsection{Insulating Taylor-Green vortex}

Lastly, we consider a modification of the well-known Taylor-Green vortex (TGV) of the the compressible Navier-Stokes equations, e.g. \cite{Shu2005}, to include magnetic fields \cite{Lee2008,Lee2010,Brachet2013} as a final example to demonstrate the increased robustness of the entropy stable DG approximation. The TGV flow was first introduced for the incompressible Navier-Stokes equations in \cite{taylor1937mechanism} as a model problem for the analysis of transition and turbulence decay in a cubic domain with periodic boundary conditions. The test case is particularly interesting, because a simple set of initial conditions evolve to include a wide range of spatial scales as well as turbulent structures. Several extensions of the TGV are available for the ideal MHD equations to model turbulent plasmas, e.g. \cite{Lee2010}. Here we consider a particular insulating version of the TGV for the resistive MHD equations. Although the original extension of the TGV was done in the ideal MHD context \cite{Lee2008}, we select the viscosity and resistivity coefficients to be \eqref{eq:resParameters}.

We adopt a modified version of the initial conditions suited for compressible flow solvers. To create an initial condition for the pressure involves solving a Poisson equation from the incompressible ideal MHD equations. Details of this process are outlined in \ref{app:ICs}. Because the definition of the pressure is unique up to a constant, we select the initial pressure such that the flow is nearly incompressible with a maximum Mach number of 0.1. The initial conditions prescribe a state in $\Omega = [0, 2\pi]^3$ with the density, velocity components, pressure, and magnetic field components defined as:
\begin{equation}\label{eq:insulTGV}
\begin{aligned}
\rho &= 1\\[0.1cm]
\spacevec{v} &= \left(\sin(x)\cos(y)\cos(z),\,-\cos(x)\sin(y)\cos(z),\,0\right)^T\\[0.1cm]
p     &= \frac{100}{\gamma} + \frac{1}{16}\left(\cos(2x)+\cos(2y)\right)\left(2+\cos(2z)\right) + \frac{1}{16}\left(\cos(4x) + \cos(4y)\right)\left(2-\cos(4z)\right)\\[0.1cm]
\spacevec{B} &= \left(\cos(2x)\sin(2y)\sin(2z),\,-\sin(2x)\cos(2y)\sin(2z),\,0\right)^T\\[0.1cm]
\end{aligned}
\end{equation}
where $\gamma = 1.4$. The initial conditions \eqref{eq:insulTGV} are evolved up to the final time $T=20.0$ such that turbulent structures can develop and evolve. We note that, for a compressible simulation the initial condition of the pressure will change, if other insulating or conducting TGV formulations from \cite{Lee2008} are chosen.

We use this turbulent test case of the insulated TGV \eqref{eq:insulTGV} to demonstrate the increased robustness of the proposed solver described in this work. To do so, we run three spatial configurations each with $64^3$ degrees of freedom (DOFs): 
\begin{itemize}
\item $N=3$ in each spatial direction with a $16\times 16 \times 16$ mesh
\item $N=7$ in each spatial direction with an $8\times 8 \times 8$ mesh
\item $N=15$ in each spatial direction with a $4\times 4 \times 4$ mesh
\end{itemize}
We perform a comparison of the standard DGSEM against the entropy stable DGSEM, both with GLM divergence cleaning. We obtain similar robustness results for each polynomial order examined in this test case and collect our findings in Table \ref{tab:crashTestsTGV}. This highlights that the entropy stable DG method with GLM divergence cleaning is more numerically stable for under-resolved turbulence computations. 
\begin{table}[h!]
	\centering
		\begin{tabular}{| l | c | c |}
		        \hline
                         \textsc{DOFs = $64^3$} & Standard DGSEM with GLM & Entropy Stable DGSEM with GLM \\
			\hline
			\hline
			$N=3$, $16 \times 16 \times 16$ mesh & \textbf{\textit{crash}} & \textbf{\textit{result}} \\
			\hline 
			$N=7$, $8 \times 8 \times 8$ mesh & \textbf{\textit{crash}} & \textbf{\textit{result}} \\
			\hline 
			$N=15$, $4 \times 4 \times 4$ mesh & \textbf{\textit{crash}} & \textbf{\textit{result}} \\
			\hline 
		\end{tabular}
		\caption{Comparison of the standard versus entropy stable high-order DGSEM for the insulating Taylor-Green vortex problem with three polynomial orders, all with $64^3$ DOFs.}\label{tab:crashTestsTGV}
\end{table}

The low Mach number insulating TGV \eqref{eq:insulTGV} considered here is used purely as a robustness test case. An analysis of the turbulence modeling capabilities obtained by the entropy stable DGSEM with GLM divergence cleaning is outside the scope of the current work, but is the focus of future research. These results underline that controlling aliasing errors in a high-order numerical approximation of the resistive GLM-MHD equations with an entropy stable formulation offers increased robustness. This reinforces previous results regarding entropy stable DG methods stabilizing under-resolved turbulence computations for the compressible Euler equations \cite{Gassner:2016ye}. 

\section{Conclusions}\label{Sec:Concl}

In this work, we have presented a novel entropy stable nodal DG scheme for the resistive MHD equations including GLM divergence cleaning. First, we have analyzed the continuous entropic properties of the underlying system in order to demonstrate that the resistive GLM-MHD equations satisfy the entropy inequality. This also provided guidance for the semi-discrete formulation to find necessary entropy stability conditions. Concerning this matter, we have always followed the methodology of constructing the discrete thermodynamic properties first for the ideal MHD system with special focus on the divergence diminishing GLM parts. Lastly, the discretization of the resistive parts in the proposed DG approximation has been shown to be entropy stable. 

We have also provided numerical results to verify the theoretical findings. In particular, we have shown, with the method of manufactured solutions, that the entropy stable DGSEM solver described in this work is high-order accurate. Next, we have verified the entropy conservative nature of the underlying scheme as well as the utility of the GLM divergence cleaning and damping term. Finally, we have demonstrated the increased robustness of the entropy stable DG approximation. The last two numerical tests reveal that the entropy stable discretization with hyperbolic divergence cleaning significantly improves the robustness in two as well as three spatial dimensions. Thus, especially for large-scale applications, e.g. in space or astrophysics, the entropy stable high-order scheme for the resistive GLM-MHD equations derived herein offers accuracy and robustness advantages compared to other state-of-the-art DG solvers.


\section*{Acknowledgements}
Gregor Gassner and Andrew Winters thank the European Research Council for funding through the ERC Starting Grant ``An Exascale aware and Un-crashable Space-Time-Adaptive Discontinuous Spectral Element Solver for Non-Linear Conservation Laws'' (\texttt{Extreme}), ERC grant agreement no. 714487.
Dominik Derigs and Stefanie Walch acknowledge the support of the Bonn-Cologne Graduate School for Physics and Astronomy (BCGS), which is funded through the Excellence Initiative, as well as the Sonderforschungsbereich (SFB) 956 on the ``Conditions and impact of star formation''.
Stefanie Walch thanks the Deutsche Forschungsgemeinschaft (DFG) for funding through the SPP 1573 ``The physics of the interstellar medium'' and the funding from the European Research Council via the ERC Starting Grant ``The radiative interstellar medium'' (\texttt{RADFEEDBACK}).
This work was partially performed on the Cologne High Efficiency Operating Platform for Sciences (CHEOPS) at the Regionales Rechenzentrum K\"oln (RRZK).

\section*{References}
\bibliography{Paper_ESDGSEM_GLMMHD}

\appendix

\section{Dissipation matrices for entropy variables}\label{Sec:DisMatrix}

In this section we explicitly state the missing block matrices necessary to define the diffusion terms for the entropy stable approximation of the resistive GLM-MHD equations from Lemma \ref{lemma1}:
\begin{equation}\label{eq:matK12}
\nineMatrix{K}_{12} = \frac{1}{w_5}\begin{pmatrix}
0 & 0 & 0 & 0 & 0 & 0 & 0 & 0 & 0 \\[0.15cm]
0 & 0 & \frac{2\viscosity}{3} & 0 & -\frac{2\viscosity w_3}{3w_5} & 0 & 0 & 0 & 0 \\[0.15cm]
0 & -\viscosity & 0 & 0 & \frac{\viscosity w_2}{w_5} & 0 & 0 & 0 & 0 \\[0.15cm]
0 & 0 & 0 & 0 & 0 & 0 & 0 & 0 & 0 \\[0.15cm]
0 & \frac{\viscosity w_3}{w_5} & -\frac{2\viscosity w_2}{3w_5} & 0 & -\frac{\viscosity w_2 w_3}{3w_5^2} & 0 & 0 & 0 & 0 \\[0.15cm]
0 & 0 & 0 & 0 & \frac{\resistivity w_6 w_7}{w_5^2} & -\frac{\resistivity w_7}{w_5} & 0 & 0 & 0 \\[0.15cm]
0 & 0 & 0 & 0 & 0 & 0 & 0 & 0 & 0 \\[0.15cm]
0 & 0 & 0 & 0 & -\frac{\resistivity w_6}{w_5} & \resistivity & 0 & 0& 0 \\[0.15cm]
0 & 0 & 0 & 0 & 0 & 0 & 0 & 0 & 0 \\[0.15cm]
\end{pmatrix}
\end{equation}

\begin{equation}
\nineMatrix{K}_{13} = \frac{1}{w_5}\begin{pmatrix}
0 & 0 & 0 & 0 & 0 & 0 & 0 & 0 & 0 \\[0.15cm]
0 & 0 & 0 & \frac{2\viscosity}{3} & -\frac{2\viscosity w_4}{3w_5} & 0 & 0 & 0 & 0 \\[0.15cm]
0 & 0 & 0 & 0 & 0 & 0 & 0 & 0 & 0 \\[0.15cm]
0 & -\viscosity & 0 & 0 & \frac{\viscosity w_2}{w_5} & 0 & 0 & 0 & 0 \\[0.15cm]
0 & \frac{\viscosity w_4}{w_5} & 0 & -\frac{2\viscosity w_2}{3 w_5} & -\frac{\viscosity w_2 w_4}{3w_5^2}+\frac{\resistivity w_6 w_8}{w_5^2} & -\frac{\resistivity w_8}{w_5} & 0 & 0 & 0 \\[0.15cm]
0 & 0 & 0 & 0 & 0 & 0 & 0 & 0 & 0 \\[0.15cm]
0 & 0 & 0 & 0 & 0 & 0 & 0 & 0 & 0 \\[0.15cm]
0 & 0 & 0 & 0 & -\frac{\resistivity w_6}{w_5} & \resistivity & 0 & 0 & 0 \\[0.15cm]
0 & 0 & 0 & 0 & 0 & 0 & 0 & 0 & 0 \\[0.15cm]
\end{pmatrix}
\end{equation}

\begin{equation}
\nineMatrix{K}_{21} = \frac{1}{w_5}\begin{pmatrix}
0 & 0 & 0 & 0 & 0 & 0 & 0 & 0 & 0 \\[0.15cm]
0 & 0 & -\viscosity & 0 & \frac{\viscosity w_3}{w_5} & 0 & 0 & 0 & 0 \\[0.15cm]
0 & \frac{2\viscosity}{3} & 0 & 0 & -\frac{2\viscosity w_2}{3 w_5} & 0 & 0 & 0 & 0 \\[0.15cm]
0 & 0 & 0 & 0 & 0 & 0 & 0 & 0 & 0 \\[0.15cm]
0 & -\frac{2\viscosity w_3}{3w_5} & \frac{\viscosity w_2}{w_5} & 0 & -\frac{\viscosity w_2 w_3}{3w_5^2} & 0 & 0 & 0 & 0 \\[0.15cm]
0 & 0 & 0 & 0 & \frac{\resistivity w_6 w_7}{w_5^2} & 0 & -\frac{\resistivity w_6}{w_5} & 0 & 0 \\[0.15cm]
0 & 0 & 0 & 0 & -\frac{\resistivity w_7}{w_5} & 0 & \resistivity & 0 & 0 \\[0.15cm]
0 & 0 & 0 & 0 & 0 & 0 & 0 & 0 & 0 \\[0.15cm]
0 & 0 & 0 & 0 & 0 & 0 & 0 & 0 & 0 \\[0.15cm]
\end{pmatrix}
\end{equation}

\begin{equation}
\resizebox{0.925\hsize}{!}{$
\nineMatrix{K}_{22} = \frac{1}{w_5}\begin{pmatrix}
0 & 0 & 0 & 0 & 0 & 0 & 0 & 0 & 0 \\[0.15cm]
0 & -\viscosity & 0 & 0 & \frac{\viscosity w_2}{w_5} & 0 & 0 & 0 & 0 \\[0.15cm]
0 & 0 & -\frac{4\viscosity}{3} & 0 & \frac{4\viscosity w_3}{3 w_5} & 0 & 0 & 0 & 0 \\[0.15cm]
0 & 0 & 0 & -\viscosity & \frac{\viscosity w_4}{w_5} & 0 & 0 & 0 & 0 \\[0.15cm]
0 & \frac{\viscosity w_2}{w_5} & \frac{4\viscosity w_3}{3 w_5} & \frac{\viscosity w_4}{w_5} & -\frac{\viscosity w_2^2}{w_5^2}-\frac{4\viscosity w_3^2}{3 w_5^2}-\frac{\viscosity w_4^2}{w_5^2}+\frac{\kappa}{R w_5}-\frac{\resistivity w_6^2}{w_5^2}-\frac{\resistivity w_8^2}{w_5^2} & \frac{\resistivity w_6}{w_5} & 0 & \frac{\resistivity w_8}{w_5} & 0 \\[0.15cm]
0 & 0 & 0 & 0 & \frac{\resistivity w_6}{w_5} & -\resistivity & 0 & 0 & 0 \\[0.15cm]
0 & 0 & 0 & 0 & 0 & 0 & 0 & 0 & 0 \\[0.15cm]
0 & 0 & 0 & 0 & \frac{\resistivity w_8}{w_5} & 0 & 0 & -\resistivity & 0 \\[0.15cm]
0 & 0 & 0 & 0 & 0 & 0 & 0 & 0 & 0 \\[0.15cm]
\end{pmatrix}$}
\end{equation}

\begin{equation}
\nineMatrix{K}_{23} = \frac{1}{w_5}\begin{pmatrix}
0 & 0 & 0 & 0 & 0 & 0 & 0 & 0 & 0 \\[0.15cm]
0 & 0 & 0 & 0 & 0 & 0 & 0 & 0 & 0 \\[0.15cm]
0 & 0 & 0 & \frac{2\viscosity}{3} & -\frac{2\viscosity w_4}{3 w_5} & 0 & 0 & 0 & 0 \\[0.15cm]
0 & 0 & -\viscosity & 0 & \frac{\viscosity w_3}{w_5} & 0 & 0 & 0 & 0 \\[0.15cm]
0 & 0 & \frac{\viscosity w_4}{w_5} & -\frac{2\viscosity w_3}{3 w_5} & -\frac{\viscosity w_3 w_4}{3 w_5^2}+\frac{\resistivity w_7 w_8}{w_5^2} & 0 & -\frac{\resistivity w_8}{w_5} & 0 & 0 \\[0.15cm]
0 & 0 & 0 & 0 & 0 & 0 & 0 & 0 & 0 \\[0.15cm]
0 & 0 & 0 & 0 & 0 & 0 & 0 & 0 & 0 \\[0.15cm]
0 & 0 & 0 & 0 & -\frac{\resistivity w_7}{w_5} & 0 & \resistivity & 0 & 0 \\[0.15cm]
0 & 0 & 0 & 0 & 0 & 0 & 0 & 0 & 0 \\[0.15cm]
\end{pmatrix}
\end{equation}

\begin{equation}
\nineMatrix{K}_{31} = \frac{1}{w_5}\begin{pmatrix}
0 & 0 & 0 & 0 & 0 & 0 & 0 & 0 & 0 \\[0.15cm]
0 & 0 & 0 & -\viscosity & \frac{\viscosity w_4}{w_5} & 0 & 0 & 0 & 0 \\[0.15cm]
0 & 0 & 0 & 0 & 0 & 0 & 0 & 0 & 0 \\[0.15cm]
0 & \frac{2\viscosity}{3} & 0 & 0 & -\frac{2\viscosity w_2}{3 w_5} & 0 & 0 & 0 & 0 \\[0.15cm]
0 & -\frac{2\viscosity w_4}{3 w_5} & 0 & \frac{\viscosity w_2}{w_5} & -\frac{\viscosity w_2 w_4}{3w_5^2}+\frac{\resistivity w_6 w_8}{w_5^2} & 0 & 0 & -\frac{\resistivity w_6}{w_5} & 0 \\[0.15cm]
0 & 0 & 0 & 0 & -\frac{\resistivity w_8}{w_5} & 0 & 0 & \resistivity & 0 \\[0.15cm]
0 & 0 & 0 & 0 & 0 & 0 & 0 & 0 & 0 \\[0.15cm]
0 & 0 & 0 & 0 & 0 & 0 & 0 & 0 & 0 \\[0.15cm]
0 & 0 & 0 & 0 & 0 & 0 & 0 & 0 & 0 \\[0.15cm]
\end{pmatrix}
\end{equation}

\begin{equation}
\nineMatrix{K}_{32} = \frac{1}{w_5}\begin{pmatrix}
0 & 0 & 0 & 0 & 0 & 0 & 0 & 0 & 0 \\[0.15cm]
0 & 0 & 0 & 0 & 0 & 0 & 0 & 0 & 0 \\[0.15cm]
0 & 0 & 0 & -\viscosity & \frac{\viscosity w_4}{w_5} & 0 & 0 & 0 & 0 \\[0.15cm]
0 & 0 & \frac{2\viscosity}{3} & 0 & -\frac{2\viscosity w_3}{3 w_5} & 0 & 0 & 0 & 0 \\[0.15cm]
0 & 0 & -\frac{2\viscosity w_4}{3 w_5} & \frac{\viscosity w_3}{w_5} & -\frac{\viscosity w_3 w_4}{3 w_5^5}+\frac{\resistivity w_7 w_8}{w_5^2} & 0 & 0 & -\frac{\resistivity w_7}{w_5} & 0 \\[0.15cm]
0 & 0 & 0 & 0 & 0 & 0 & 0 & 0 & 0 \\[0.15cm]
0 & 0 & 0 & 0 & -\frac{\resistivity w_8}{w_5} & 0 & 0 & \resistivity & 0 \\[0.15cm]
0 & 0 & 0 & 0 & 0 & 0 & 0 & 0 & 0 \\[0.15cm]
0 & 0 & 0 & 0 & 0 & 0 & 0 & 0 & 0 \\[0.15cm]
\end{pmatrix}
\end{equation}

\begin{equation}\label{eq:matK33}
\resizebox{0.925\hsize}{!}{$
\nineMatrix{K}_{33} = \frac{1}{w_5} \begin{pmatrix}
0 & 0 & 0 & 0 & 0 & 0 & 0 & 0 & 0 \\[0.15cm]
0 & -\viscosity & 0 & 0 & \frac{\viscosity w_2}{w_5} & 0 & 0 & 0 & 0 \\[0.15cm]
0 & 0 & -\viscosity & 0 & \frac{\viscosity w_3}{w_5} & 0 & 0 & 0 & 0 \\[0.15cm]
0 & 0 & 0 & -\frac{4\viscosity}{3} & \frac{4\viscosity w_4}{3 w_5} & 0 & 0 & 0 & 0 \\[0.15cm]
0 & \frac{\viscosity w_2}{w_5} & \frac{\viscosity w_3}{w_5} & \frac{4\viscosity w_4}{3 w_5} & -\frac{\viscosity w_2^2}{w_5^2}-\frac{\viscosity w_3^2}{w_5^2}-\frac{4\viscosity w_4^2}{3w_5^2}+\frac{\kappa}{Rw_5^2}-\frac{\resistivity w_6^2}{w_5^2}-\frac{\resistivity w_6^2}{w_5^2} & \frac{\resistivity w_6}{w_5} & \frac{\resistivity w_7}{w_5} & 0 & 0 \\[0.15cm]
0 & 0 & 0 & 0 & \frac{\resistivity w_6}{w_5} & -\resistivity & 0 & 0 & 0 \\[0.15cm]
0 & 0 & 0 & 0 & \frac{\resistivity w_7}{w_5} & 0 & -\resistivity & 0 & 0 \\[0.15cm]
0 & 0 & 0 & 0 & 0 & 0 & 0 & 0 & 0 \\[0.15cm]
0 & 0 & 0 & 0 & 0 & 0 & 0 & 0 & 0 \\[0.15cm]
\end{pmatrix}$}
\end{equation}

\section{Proof for MHD volume contribution}\label{Sec:MHD_vol}

We show in this section that the property \eqref{eq:MHDVolTerms}, reproduced here for convenience,
\begin{equation}\label{eq:volInProof}
\iprodN{\spacevec{\mathbb{D}}\bigcontravec F^{\ec,\text{MHD}},\statevec W} + \iprodN{\JanD\spacevec{\mathbb{D}}^{S}\contraspacevec{B},\statevec{W}} = 0,
\end{equation}
is satisfied. To do so we first expand each of the volume contribution from the advective terms
\begin{equation}
\resizebox{0.925\hsize}{!}{$
\iprodN{\spacevec{\mathbb{D}}\bigcontravec F^{\ec,\text{MHD}},\statevec{W}} =
 \sum\limits_{i,j=0}^{N}\omega_{i}\omega_j\statevec{W}^T_{ij}\left[\frac{\Delta x}{2}\sum\limits_{m=0}^N 2\dmat_{im}\statevec{F}_1^{\ec,\text{MHD}}(\statevec{U}_{ij}, \statevec{U}_{mj}) + \frac{\Delta x}{2}\sum\limits_{m=0}^N 2\dmat_{jm}\statevec{F}_2^{\ec,\text{MHD}}(\statevec{U}_{ij}, \statevec{U}_{im})\right],$}
\end{equation}
where we peeled the constant mapping term $\frac{\Delta x}{2}$ out of the entropy conservative fluxes for convenience. Next, we expand the volume contribution from the non-conservative term
\begin{equation}\label{eq:nonConsProofAp}
\begin{aligned}
\iprodN{\JanD \spacevec{\mathbb{D}}^{\text{S}}\contraspacevec{B},\statevec{W}}=
\sum\limits_{i,j=0}^{N}\omega_{i}\omega_j\statevec{W}^T_{ij}&\left[\frac{\Delta x}{2}\sum_{m=0}^N \dmat_{im}\JanD_{ij}\left(B_1\right)_{mj}
+ \frac{\Delta x}{2}\sum_{m=0}^N \dmat_{jm}\,\JanD_{ij}\left(B_2\right)_{im}\right],
\end{aligned}
\end{equation}
again, the constant mapping is factored out for convenience.

We focus on the $\xi-$direction term of the volume integral approximations, as the $\eta-$direction is done in an analogous manner. The sum can be written in terms of the SBP matrix \eqref{SBP}, $\qmat_{im}=\omega_{i}\dmat_{im}$,
\begin{equation}
\resizebox{0.925\hsize}{!}{$
\frac{\Delta x}{2}\sum\limits_{j=0}^{N}\omega_{j}\sum\limits_{i=0}^N\statevec{W}^T_{ij}\sum\limits_{m=0}^N 2\omega_i\dmat_{im}\statevec{F}_1^{\ec,\text{MHD}}(\statevec{U}_{ij}, \statevec{U}_{mj})
=\frac{\Delta x}{2}\sum\limits_{j=0}^{N}\omega_{j}\sum\limits_{i=0}^N\statevec{W}^T_{ij}\sum\limits_{m=0}^N 2\qmat_{im}\statevec{F}_1^{\ec,\text{MHD}}(\statevec{U}_{ij}, \statevec{U}_{mj}).$}
\end{equation}
We use the summation-by-parts property $2\qmat_{im} = \qmat_{im} - \qmat_{mi} + \bmat_{im}$, perform a reindexing of $i$ and $m$ to incorperate the $\qmat_{mi}$ term and use the fact that $\statevec{F}_1^{\ec,\text{MHD}}(\statevec{U}_{ij}, \statevec{U}_{mj})$ is symmetric with respect to the index $i$ and $m$ to rewrite the $\xi-$direction contribution to the volume integral approximation as
\begin{equation}\label{eq:summ1Ap}
\begin{aligned}
\sum\limits_{i=0}^N\statevec{W}^T_{ij}\sum_{m=0}^N 2\qmat_{im}\statevec{F}_1^{\ec,\text{MHD}}(\statevec{U}_{ij}, \statevec{U}_{mj})
&= \sum_{i,m=0}^N\statevec W^T_{ij}(\qmat_{im} - \qmat_{mi} + \bmat_{im})\statevec{F}_1^{\ec,\text{MHD}}(\statevec{U}_{ij},\statevec{U}_{mj})\\
&= \sum_{i,m=0}^N\qmat_{im}\left(\statevec W_{ij} - \statevec W_{mj}\right)^T
\statevec{F}_1^{\ec,\text{MHD}}(\statevec{U}_{ij},\statevec{U}_{mj}) \\
&\qquad\qquad\quad + \bmat_{im}\,\statevec W^T_{ij}\statevec{F}_1^{\ec,\text{MHD}}(\statevec{U}_{ij},\statevec{U}_{mj}).
\end{aligned}
\end{equation}

We have divided the entropy flux potentials into Euler, ideal MHD and GLM components \eqref{eq:EulerEntFluxPot}-\eqref{eq:GLMEntFluxPot}. Because the proof at hand only considers ideal MHD terms, we are only concerned with \eqref{eq:MHDEntFluxPot}
\begin{equation}
\Psi_1^{\text{MHD}}= \statevec{w}^T\statevec{f}_1^{\text{MHD}} + \wDotJan B_1
\end{equation}
and the accompanying entropy conservation condition for the MHD part of \eqref{eq:MHDPartDiscrete}
\begin{equation}\label{eq:entConditionInProofAp}
\jump{\statevec{w}}^T\statevec{f}_1^{\ec,\text{MHD}} = \jump{\Psi_1^{\text{MHD}}} - \avg{B_1}\jump{\wDotJan}.
\end{equation}
We apply the form of \eqref{eq:entConditionInProofAp} to rewrite \eqref{eq:summ1Ap}
\begin{equation}\label{eq:moreBsAp}
\left(\statevec W_{ij} - \statevec W_{mj}\right)^T\statevec{F}_1^{\ec,\text{MHD}}(\statevec{U}_{ij},\statevec{U}_{mj}) = \left(\Psi^{\text{MHD}}_1\right)_{ij}-\left(\Psi^{\text{MHD}}_1\right)_{mj}
-\frac{1}{2}\left(\left(B_1\right)_{ij}+\left(B_1\right)_{mj}\right)\left(\wDotJan_{ij}-\wDotJan_{mj}\right).
\end{equation}
Furthermore, note that the entries of the boundary matrix $\bmat$ are only non-zero when $i=m=0$ or $i=m=N$, so
\begin{equation}\label{eq:boundaryBsAp}
\bmat_{im}\statevec{W}^T_{ijk}\statevec{F}_1^{\ec,\text{MHD}}(\statevec{U}_{ij},\statevec{U}_{mj}) = \bmat_{im}\left(\left(\Psi_1^{\text{MHD}}\right)_{ij} - \wDotJan_{ij}\left(B_1\right)_{ij}\right).
\end{equation}

We substitute \eqref{eq:moreBsAp} and \eqref{eq:boundaryBsAp} into the final line of \eqref{eq:summ1Ap} to find
\begin{equation}\label{eq:awfulSumAp}
\resizebox{0.925\hsize}{!}{$
\begin{aligned}
\sum\limits_{i=0}^N\statevec W^T_{ij}\sum_{m=0}^N 2\qmat_{im}\statevec{F}_1^{\ec,\text{MHD}}(\statevec{U}_{ij}, \statevec{U}_{mj})
&=\sum\limits_{i,m=0}^N\qmat_{im}\left[\left({\Psi}_1^{\text{MHD}}\right)_{ij} - \left({\Psi}_1^{\text{MHD}}\right)_{mj} - \frac{1}{2}\left(\left(B_1\right)_{ij} + \left(B_1\right)_{mj}\right)\left(\wDotJan_{ij}-\wDotJan_{mj}\right)\right]\\
&\qquad\qquad\qquad + \bmat_{im}\left[\left(\Psi_1^{\text{MHD}}\right)_{ij} - \wDotJan_{ij}\left(B_1\right)_{ij}\right].
\end{aligned}$}
\end{equation}
We will examine the terms of the sum \eqref{eq:awfulSumAp} systematically from left to right. Now, because the derivative of a constant is zero (i.e. the rows of $\qmat$ sum to zero),
\begin{equation}\label{eq:B11Ap}
\sum\limits_{i,m=0}^N\qmat_{im}\left(\Psi_1^{\text{MHD}}\right)_{ij} = \sum\limits_{i=0}^N\left(\Psi_1^{\text{MHD}}\right)_{ij}\sum\limits_{m=0}^N\qmat_{im} = 0.
\end{equation}
Next, on the second term, we use the summation by parts property \eqref{SBP}, and reindex on the $\qmat_{mi}$ term to rewrite
\begin{equation}\label{eq:B12Ap}
\begin{aligned}
-\!\!\!\sum\limits_{i,m=0}^N\qmat_{im}\left(\Psi_1^{\text{MHD}}\right)_{mj}
&= -\!\!\!\sum\limits_{i,m=0}^N(\bmat_{im}-\qmat_{mi})\left(\Psi_1^{\text{MHD}}\right)_{mj} \\
&=-\!\!\!\sum\limits_{i,m=0}^N(\bmat_{im}-\qmat_{im})\left(\Psi_1^{\text{MHD}}\right)_{ij} \\
&=-\!\!\!\sum\limits_{i,m=0}^N\bmat_{im}\left(\Psi_1^{\text{MHD}}\right)_{ij} + \sum\limits_{i,m=0}^N\qmat_{im}\left(\Psi_1^{\text{MHD}}\right)_{ij} \\
&=-\!\!\!\sum\limits_{i,m=0}^N\bmat_{im}\left(\Psi_1^{\text{MHD}}\right)_{ij} + \sum\limits_{i=0}^N\left(\Psi_1^{\text{MHD}}\right)_{ij}\sum\limits_{m=0}^N \qmat_{im}\\
&=-\!\!\!\sum\limits_{i,m=0}^N\bmat_{im}\left(\Psi_1^{\text{MHD}}\right)_{ij} + 0,
\end{aligned}
\end{equation}
where, again, one term cancels due to consistency of the matrix $\qmat$.

We come next to the terms involving $B_1$ and $\wDotJan$ in \eqref{eq:awfulSumAp}. We leave these terms grouped for convenience. First, we expand to find
\begin{equation}\label{eq:expandGrossAp}
\begin{aligned}
\sum\limits_{i,m=0}^N\qmat_{im}&\left(-\frac{1}{2}\left(\left(B_1\right)_{ij}+\left(B_1\right)_{mj}\right)\left(\wDotJan_{ij}-\wDotJan_{mj}\right)\right) \\
&=-\frac{1}{2}\sum\limits_{i,m=0}^N \qmat_{im}\left(\wDotJan_{ij}\left(B_1\right)_{ij} + \wDotJan_{ij}\left(B_1\right)_{mj} - \wDotJan_{mj}\left(B_1\right)_{ij} - \wDotJan_{mj}\left(B_1\right)_{mj}\right).
\end{aligned}
\end{equation}
We examine each term from \eqref{eq:expandGrossAp}: the first term uses the consistency of $\qmat$ to become zero, the second term is left alone, the third term reindexes $i$ and $m$ and the fourth term uses the SBP property to obtain
\begin{equation}\label{eq:expandGross2Ap}
\begin{aligned}
-\frac{1}{2}\sum\limits_{i,m=0}^N& \qmat_{im}\left(\wDotJan_{ij}\left(B_1\right)_{ij} + \wDotJan_{ij}\left(B_1\right)_{mj} - \wDotJan_{mj}\left(B_1\right)_{ij} - \wDotJan_{mj}\left(B_1\right)_{mj}\right)\\
&=0 - \frac{1}{2}\sum\limits_{i,m=0}^N(\qmat_{im}-\qmat_{mi})\wDotJan_{ij}\left(B_1\right)_{mj} + \frac{1}{2}\sum\limits_{i,m=0}^N(\bmat_{im}-\qmat_{mi})\wDotJan_{mj}\left(B_1\right)_{mj}.
\end{aligned}
\end{equation}
Next, we use the SBP property on the $\qmat_{mi}$ term in the second part of \eqref{eq:expandGross2Ap}, reindex in $i$ and $m$ on the third term to get and use the consistency of the $\qmat$ twice to find
\begin{equation}\label{eq:expandGross3Ap}
\begin{aligned}
-\frac{1}{2}\sum\limits_{i,m=0}^N& \qmat_{im}\left(\wDotJan_{ij}\left(B_1\right)_{ij} + \wDotJan_{ij}\left(B_1\right)_{mj} - \wDotJan_{mj}\left(B_1\right)_{ij} - \wDotJan_{mj}\left(B_1\right)_{mj}\right)\\
&=-\frac{1}{2}\sum\limits_{i=0}^N\wDotJan_{ij}\left(B_1\right)_{ij}\sum\limits_{m=0}^N\qmat_{im}-\frac{1}{2}\sum\limits_{i,m=0}^N(2\qmat_{im}-\bmat_{im})\wDotJan_{ij}\left(B_1\right)_{mj}\\[0.05cm]
&\qquad\qquad\qquad\qquad\qquad\qquad\quad + \frac{1}{2}\sum\limits_{i,m=0}^N\bmat_{im}\wDotJan_{ij}\left(B_1\right)_{ij}-\frac{1}{2}\sum\limits_{i=0}^N\wDotJan_{ij}\left(B_1\right)_{ij}\sum\limits_{m=0}^N\qmat_{im}\\[0.05cm]
&= 0 + \sum\limits_{i,m=0}^N\bmat_{im}\wDotJan_{ij}\left(B_1\right)_{ij} - \sum\limits_{i,m=0}^N\qmat_{im}\wDotJan_{ij}\left(B_1\right)_{mj} + 0
\end{aligned}
\end{equation}

Combining the results of \eqref{eq:awfulSumAp}, \eqref{eq:B11Ap}, \eqref{eq:B12Ap} and \eqref{eq:expandGross3Ap}, we have
\begin{equation}\label{eq:awfulSumReduxAp}
\resizebox{0.925\hsize}{!}{$
\begin{aligned}
\sum\limits_{i=0}^N&\statevec W^T_{ij}\sum_{m=0}^N 2\qmat_{im}\statevec{F}_1^{\ec,\text{MHD}}(\statevec{U}_{ij}, \statevec{U}_{mj}) \\
&=-\!\!\!\sum\limits_{i,m=0}^N\bmat_{im}\left(\Psi_1^{\text{MHD}}\right)_{ij} - \!\!\!\sum\limits_{i,m=0}^N\qmat_{im}\wDotJan_{ij}\left(B_1\right)_{mj}
+ \sum\limits_{i,m=0}^N\bmat_{im}\wDotJan_{ij}\left(B_1\right)_{ij} + \sum\limits_{i,m=0}^N\bmat_{im}\left[\left(\Psi_1^{\text{MHD}}\right)_{ij} - \wDotJan_{ij}\left(B_1\right)_{ij}\right]\\[0.05cm]
&= -\!\!\!\sum\limits_{i,m=0}^N\qmat_{im}\wDotJan_{ij}\left(B_1\right)_{mj}
\end{aligned}$}
\end{equation}

We are now prepared to revisit the contributions from the non-conservative volume terms \eqref{eq:nonConsProofAp}. In the the $\xi-$direction this contribution takes the form
\begin{equation}\label{eq:nonConsProof1Ap}
\begin{aligned}
\sum\limits_{i,j=0}^{N}\omega_{i}\omega_{j}\statevec{W}^T_{ij} \sum_{m=0}^N \dmat_{im}\JanD_{ij}\left(B_1\right)_{mj}
=&\sum\limits_{j=0}^N\omega_{j}\sum_{i,m=0}^N\omega_i\dmat_{im}\statevec{W}^T_{ij}\JanD_{ij}\left(B_1\right)_{mj} \\
=&\sum\limits_{j=0}^N\omega_{j}\sum_{i,m=0}^N\qmat_{im}\statevec{W}^T_{ij}\JanD_{ij}\left(B_1\right)_{mj} \\
=&\sum\limits_{j=0}^N\omega_{j}\sum_{i,m=0}^N\qmat_{im}\wDotJan_{ij}\left(B_1\right)_{mj},
\end{aligned}
\end{equation}
where we have used the definition of the SBP matrix and the property \eqref{eq:contractSource} contracting the non-conservative term into entropy space.
Comparing the result \eqref{eq:nonConsProof1Ap} and the term of \eqref{eq:awfulSumReduxAp} we see that they cancel each other when added together. Thus, the contribution in the $\xi-$direction is
\begin{equation}\label{eq:awfulSumRedux3Ap}
\sum\limits_{j=0}^N\omega_{j}\sum\limits_{i=0}^N\statevec W^T_{ij}\left[\sum_{m=0}^N 2\qmat_{im}\statevec{F}_1^{\ec,\text{MHD}}(\statevec{U}_{ij}, \statevec{U}_{mj})
+ \sum_{m=0}^N\qmat_{im}\wDotJan_{ij}\left(B_1\right)_{mj}\right] = 0.
\end{equation}

An analogous result to \eqref{eq:awfulSumRedux3Ap} holds in the $\eta-$direction of the volume integral approximations, leading to the desired result
\begin{equation}
\iprodN{\spacevec{\mathbb{D}}\bigcontravec F^{\ec,\text{MHD}},\statevec W} + \iprodN{\JanD\spacevec{\mathbb{D}}^S\contraspacevec{B},\statevec{W}}=0.
\end{equation}

\section{Pressure initial condition for insulating compressible Taylor-Green vortex}\label{app:ICs}

We start from the incompressible ideal MHD equations
\begin{equation}
\begin{aligned}
\pderivative{\spacevec{v}}{t} + \spacevec{v}\cdot\spacevec{\nabla}\spacevec{v} &= -\frac{1}{\rho}\spacevec{\nabla}p+\frac{1}{\rho}\spacevec{B}\cdot\spacevec{\nabla}\spacevec{B}\\[0.1cm]
\pderivative{\spacevec{B}}{t} + \spacevec{v}\cdot\spacevec{\nabla}\spacevec{B} &= \spacevec{B}\cdot\spacevec{\nabla}\spacevec{v}\\[0.1cm]
\spacevec{\nabla}\cdot\spacevec{v} &= 0,\quad \spacevec{\nabla}\cdot\spacevec{B} = 0\\[0.1cm]
\end{aligned}
\end{equation}
where, note, the velocity and magnetic fields are assumed to be solenoidal. We then take the divergence of the momentum and induction equations. The divergence of the induction equation immediately becomes zero because of the solenoidal assumptions. After repeated use of the solenoidal assumptions the divergence of the momentum equations yields a Poisson equation for the pressure
\begin{equation}\label{eq:Poisson1}
\begin{aligned}
\spacevec{\nabla}^2 p &= -\rho\left\{\left(\pderivative{v_1}{x}\right)^2+\left(\pderivative{v_2}{y}\right)^2+\left(\pderivative{v_3}{z}\right)^2 + 2 \left[\pderivative{v_2}{x}\pderivative{v_1}{y} + \pderivative{v_3}{x}\pderivative{v_1}{z} + \pderivative{v_3}{y}\pderivative{v_2}{z}\right]\right\}\\
&\qquad
+
\left(\pderivative{B_1}{x}\right)^2+\left(\pderivative{B_2}{y}\right)^2+\left(\pderivative{B_3}{z}\right)^2 + 2 \left[\pderivative{B_2}{x}\pderivative{B_1}{y} + \pderivative{B_3}{x}\pderivative{B_1}{z} + \pderivative{B_3}{y}\pderivative{B_2}{z}\right],
\end{aligned}
\end{equation}
where
\begin{equation}
\spacevec{\nabla}^2 p = \frac{\partial^2 p}{\partial x^2} + \frac{\partial^2 p}{\partial y^2} + \frac{\partial^2 p}{\partial z^2}.
\end{equation}
The initial conditions of an insulating Taylor-Green vortex \cite{Lee2008} for the density, velocity and magnetic fields, reproduced here for convenience, are
\begin{equation}\label{eq:ICsProof}
\begin{aligned}
\rho &= 1, \\[0.1cm]
\spacevec{v} &= \left(\sin(x)\cos(y)\cos(z),\,-\cos(x)\sin(y)\cos(z),0\right)^T, \\[0.1cm]
\spacevec{B} &= \left(\cos(2x)\sin(2y)\sin(2z),\,-\sin(2x)\cos(2y)\sin(2z),\,0\right)^T \\[0.1cm]
\end{aligned}
\end{equation}
The assumed initial conditions simplify the expression \eqref{eq:Poisson1} to be
\begin{equation}\label{eq:PoisssonPressure}
\spacevec{\nabla}^2 p = -\left\{\left(\pderivative{v_1}{x}\right)^2+\left(\pderivative{v_2}{y}\right)^2+ 2\pderivative{v_2}{x}\pderivative{v_1}{y}\right\}
+
\left(\pderivative{B_1}{x}\right)^2+\left(\pderivative{B_2}{y}\right)^2 + 2\pderivative{B_2}{x}\pderivative{B_1}{y}.
\end{equation}
Applying the form of the initial conditions \eqref{eq:ICsProof} the Poisson equation for the pressure \eqref{eq:PoisssonPressure} is
\begin{equation}
\spacevec{\nabla}^2 p = -\cos^2(z)\left(\cos(2x)+\cos(2y)\right)-4\sin^2(2z)\left(\cos(4x)+\cos(4y)\right).
\end{equation}
From this equation it is straightforward to derive an intial pressure, assuming periodic boundary conditions, to be
\begin{equation}
p = C + \frac{1}{16}\left(\cos(2x)+\cos(2y)\right)\left(2+\cos(2z)\right) + \frac{1}{16}\left(\cos(4x) + \cos(4y)\right)\left(2-\cos(4z)\right).
\end{equation}
We note that the pressure is uniquely defined up to a constant $C$. The value of $C$ is selected to specify the initial Mach number of the flow configuration.
\end{document}